\documentclass[11pt]{amsart}

\advance\hoffset by -0.5 truecm

\usepackage{amssymb}
\usepackage{amsthm}
\usepackage{amscd}
\usepackage{verbatim}
\usepackage{epsfig}
\usepackage{euscript}
\usepackage[colorlinks=true]{hyperref}
\hypersetup{urlcolor=blue, citecolor=red}
\usepackage{color}
\usepackage[shortlabels]{enumitem}

\newtheorem{Theorem}{Theorem}[section]
\newtheorem{Corollary}[Theorem]{Corollary}
\newtheorem{Lemma}[Theorem]{Lemma}
\newtheorem{Proposition}[Theorem]{Proposition}
\theoremstyle{definition}
\newtheorem{Definition}[Theorem]{Definition}

\theoremstyle{remark}
\newtheorem{remark}{Remark}[section]

\def \R{\mathbb R}
\def \N{\mathbb N}
\def \X{ \mathcal X}
\def \C{\mathbb{C}}
\def \p{\partial}

\def \O{\Omega}

\def \source{\mathcal S}
\def \rec{\mathcal R}
\newcommand{\TTD}[2]{F_{#1}^{#2}}
\def \dummy{\;\cdot \;}

\newcommand{\supp}{\operatorname{supp}}

\newcommand{\diam}{\operatorname{diam}}

\newcommand{\id}{\operatorname{I}}
\newcommand{\divg}{\operatorname{div}}

\newcommand{\Char}{\operatorname{Char}}
\newcommand{\sgn}{\operatorname{sgn}}
\newcommand{\eps}{\varepsilon}
\newcommand{\abs}[1]{\left\lvert#1\right\rvert}

\title{Determination of Wave Speed from Point Sources}

\author[Ilmavirta]{Joonas Ilmavirta}
\address{J. Ilmavirta, Department of Mathematics and Statistics, University of Jyv\"askyl\"a, Jyv\"askyl\"a, 40014, Finland} \email{joonas.ilmavirta@jyu.fi}

\author[Saksala]{Teemu Saksala}
\address{T. Saksala, Department of Mathematics\\
North Carolina State University, Raleigh\\
NC 27695, USA}
\email{tssaksal@ncsu.edu}

\author[Tarikere]{Ashwin Tarikere}
\address{A. Tarikere, Department of Mathematics and Statistics, University of Jyv\"askyl\"a, Jyv\"askyl\"a, 40014, Finland} \email{ashtarik@jyu.fi}

\author[Uhlmann]{Gunther Uhlmann}
\address{G. Uhlmann, Department of Mathematics, University of Washington, Seattle, WA 98195, USA} \email{gunther@math.washington.edu }


\begin{document}

\begin{abstract}
We consider the wave equation with variable wave speed in Euclidean space, with point sources modeled by Dirac delta initial displacement data. We distinguish three subsets of the ambient space: (1) the source set where the delta initial conditions are supported, (2) the unknown set where the wave speed is not known a priori, and (3) the receiver set where the waves are measured. The inverse problem is to reconstruct the wave speed uniquely in the unknown set from an unlabeled collection of waves generated by point sources and measured in the receiver set. We use propagation of singularities and sharp finite speed of propagation to reduce these data to geometric travel-time data, whose form depends on the relative positions of the three sets. We give two scenarios in which this procedure leads to unique determination of the wave speed.
\end{abstract}

\maketitle

\section{Introduction}

Can the wave speed inside the Earth be uniquely determined from passive measurements of earthquakes?
We formulate this problem for scalar waves, with earthquakes modeled as delta initial conditions. In this setting, we show that uniqueness holds \emph{without any geometric assumptions} if sources are present at every point in the set where we seek to recover the wave speed.

Most inverse problems considered in wave-based imaging have used active measurements:
the experimenter is free to send any wave from the boundary and measure the resulting responses. These data are encoded in the hyperbolic Dirichlet-to-Neumann map.
In our setting, no such map is available because the measurements are passive.
Since we cannot construct special solutions with prescribed properties, we use a different microlocal argument to obtain travel-time data that make the resulting geometric inverse problem more accessible.

Our pivotal result, Theorem~\ref{thm:reduction}, is a reduction from analytic data to geometric data.
Different arrangements of the unknown set (where the speed is not known a priori), the source set (where the earthquakes originate), and the receiver set (where the waves are observed) lead to different geometric inverse problems.
In Corollary~\ref{cor:sources_and_unknowns_agree_receivers anywhere}, we obtain travel-time data, while in Corollary~\ref{cor:enclosure_sources_enclosure_measurement}, we obtain the boundary distance function.
Both lead to unique determination of the wave speed, with no additional geometric assumptions in the first case.

To set things up precisely, let $M>1$, $n\in \{2,3,\ldots\}$, and $c \in C^\infty(\R^n)$ be such that
\begin{equation}
\label{eq:uniform_sound_speed}
M > c(x)>\frac{1}{M} >0 \qquad \text{for all } x\in \R^n.
\end{equation}

We choose nonempty open sets $\source, \rec \subset \R^n$, called the \emph{source} and \emph{receiver} sets, respectively. For each $y\in \R^n$ and $A \in \R \setminus \{0\}$, we let $u^{y,A}$ be the solution of the hyperbolic initial value problem
\begin{equation}\label{ivp}
    \begin{cases}
    (\partial_t^2-c(x)^{\alpha}\divg_x(c(x)^{2-\alpha}\nabla_x))u(t,x) = 0, & (t,x) \in \R\times \R^n, \\
    u(0,x) = A\delta_{y}, & x \in \R^n,\\
    \partial_t u(0,x) = 0, & x\in \R^n.
    \end{cases}
\end{equation}
Here, $\alpha \in \R$ is a fixed parameter.
The choice $\alpha = 2$ gives the usual wave operator for the wave speed $c\in C^\infty(\R^n)$, while $\alpha =n$ gives the Riemannian wave operator for a conformally Euclidean metric.
The principal symbol of this hyperbolic operator is independent of $\alpha$.

For a fixed $T>0$, we make the wave-based measurements
\begin{equation}
\label{eq:wave_data_def}
D_c^{\source, \rec, T} = \{u^{y,A}|_{(0,T)\times \rec}\ : \ y \in \source, \ A \in \R\setminus \{0\}\},
\end{equation}
and consider the following inverse problem: \emph{Let $\Omega \subset \R^n$ be a bounded domain with smooth boundary ($n\geq 2$). Under which geometric conditions on $T$, $\O$, $\source$, and $\rec$ do the data $D_c^{\source, \rec, T}$ uniquely determine the wave speed $c$ inside $\Omega$?}

Henceforth, we let $e$ denote the standard Euclidean metric on $\R^n$, and all norms and inner products in $\R^n$, unless otherwise indicated, will be Euclidean. Given a wave speed function $c\in C^\infty(\R^n)$, we write
\begin{equation}
\label{eq:conf_euc_metric}
g_c = c^{-2}(x) e
\end{equation}
for the conformally Euclidean Riemannian metric corresponding to $c$. Observe that if $c$ satisfies \eqref{eq:uniform_sound_speed} for some $M>1$, then $g_c$ is a complete Riemannian metric on $\R^n$. We also use the notation $d_c$ for the distance function associated with the metric $g_c$.

We define the \emph{travel-time function} of the metric $g_c$ as
\begin{equation}
\begin{split}
    \label{eq:TTD}
    &\TTD{c}{\source, \rec}\colon \source \to C(\rec),
    \\
    &\TTD{c}{\source, \rec}(y)=d_{c}(y,\dummy)\big|_{\rec}.
\end{split}
\end{equation}
This function measures the Riemannian distance $d_c$ with respect to $g_c$ from $y$ to every $z \in \rec$, and its image set $\TTD{c}{\source, \rec}(\source)$ is called the \emph{travel-time data}.

\begin{Theorem}\label{thm:reduction}
    Let $M>1$, $n\in \{2,3,\ldots\}$, and $\source,\rec \subset \R^n$ be nonempty open sets. If $c_1, c_2 \in C^\infty(\R^n)$ satisfy the estimate \eqref{eq:uniform_sound_speed} and $T>M\diam_e(\source \cup \rec)$
    then
    \begin{equation}
        \label{eq:wave_data}
        D_{c_1}^{\source, \rec, T}=D_{c_2}^{\source, \rec,T}
    \end{equation}
    implies
\begin{equation}
    \label{eq:TTT_agree}
\TTD{c_1}{\source, \rec}(\source)=\TTD{c_2}{\source, \rec}(\source).
\end{equation}
\end{Theorem}
The proof proceeds by showing that for every $y \in \source$ and $z\in \rec$, the distance $d_{c_i}(y,z)$ is precisely the time it takes for the wave $u_i^{y,A}$ to arrive at $z$, where $u_i^{y,A}$ is the unique solution of \eqref{ivp} with $c=c_i$. We note that the equation \eqref{eq:wave_data} does not mean that for each $(y,A)$ the waves $u^{y,A}_1$ and $u^{y,A}_2$ agree in $(0,T)\times \rec$, but rather that for each $(y_1,A_1)$ there is a pair $(y_2,A_2)$ such that the waves $u^{y_1,A_1}_1$ and $u^{y_2,A_2}_2$ agree in $(0,T)\times \rec$. Similarly,
\eqref{eq:TTT_agree} does not mean that
\[
d_{c_1}(y,z) = d_{c_2}(y,z) \qquad \text{for all }y \in \source, z \in \rec
\]
but rather that for each $y_1 \in \source$, there exists $y_2\in \source$ such that $d_{c_1}(y_1,z) = d_{c_2}(y_2,z)$ for all $z \in \rec$. Furthermore, Theorem~\ref{thm:reduction} requires no geometric assumptions beyond \eqref{eq:uniform_sound_speed}, showing that the PDE-based inverse problem for the waves \eqref{eq:wave_data} can always be reduced to a geometric inverse problem for travel times \eqref{eq:TTT_agree}.

Knowledge of the travel times \eqref{eq:TTT_agree} is sufficient to recover the wave speed uniquely in the receiver set and up to a conformal diffeomorphism in the source set. We formulate this result in a more general setting than is needed for this paper, as it might be of independent interest to the reader.
\begin{Proposition}
\label{prop:geom}
    Let $X$ be a smooth manifold without boundary equipped with two complete Riemannian metrics $g_1$ and $g_2$. Let $\source,\rec \subset X$ be open and bounded sets (with respect to $g_1$ or $g_2$) with smooth boundaries. If
\begin{equation}
\label{eq:geom_data}
\TTD{g_1}{\source, \rec}(\source)=\TTD{g_2}{\source, \rec}(\source)
\end{equation}
then $g_1=g_2$ in $\bar \rec$ and there is a homeomorphism $\varphi \colon \bar \source \to \bar \source$ that is a diffeomorphism of $\source$ with the property $g_1=\varphi^\ast g_2$ therein.
\end{Proposition}

In order to recover the wave speed in the domain $\O$ from \eqref{eq:TTD}, we need to impose geometric conditions on $T$, $\O$, $\source$, and $\rec$. We explore three cases in which this is possible, without claiming that they are exhaustive. Recovery from travel-time data is impossible if, for instance, no distance-minimizing geodesic from $\source$ to $\rec$ goes through $\O$.
The common feature in the cases below is that two of the three sets $\source,\rec,\O$ agree.

\begin{Corollary}
    \label{cor:sources_and_unknowns_agree_receivers anywhere}
Let $M>1$, $n\in \{2,3,\ldots\}$, and $\source, \rec \subset \R^n$ be smooth bounded domains. Suppose that $\O=\source$ or $\O=\rec$, $c_1, c_2 \in C^\infty(\R^n)$ satisfy the estimate \eqref{eq:uniform_sound_speed}, and $c_1|_{\R^n\setminus \Omega} = c_2|_{\R^n\setminus \Omega}$.
    If $T>M\diam_e(\rec \cup \source)$ and $D_{c_1}^{\source, \rec, T}=D_{c_2}^{\source, \rec,T}$ then $c_1\equiv c_2$ in $\R^n$.
\end{Corollary}

\begin{Corollary}
\label{cor:enclosure_sources_enclosure_measurement}
Let $M>1$, $n\in \{2,3,\ldots\}$, $\O \subset \R^n$ be a smooth, bounded, and convex domain, $\epsilon>0$, and $\rec=\source=B_e(\O,\epsilon)\setminus \bar \O$. Suppose $c_1, c_2 \in C^\infty(\R^n)$ satisfy \eqref{eq:uniform_sound_speed}, $T>M\diam_e(\rec)$, $(\bar \O,g_{c_i})$ is simple for both $i \in\{1,2\}$, and
\begin{equation}
\label{eq:ext_assumptions_for_ci}
\text{the ambient and intrinsic $d_{c_i}$-distances agree on $\bar \O \times \bar \O$}.
\end{equation}
If $D_{c_1}^{\source, \rec, T}=D_{c_2}^{\source, \rec,T}$ then $c_1\equiv c_2$ in $\bar \Omega$.
\end{Corollary}

\subsection{Related work}
Most previous work on inverse problems for the wave equation has used boundary data. The boundary control method, originally developed by Belishev \cite{Belishev1987} for the acoustic wave equation and later expanded to Riemannian manifolds by Belishev and Kurylev \cite{BelishevKurylev1992}, reconstructs a Riemannian metric $g$ on a smooth bounded domain $\Omega$ from the hyperbolic Dirichlet-to-Neumann map of the operator $\partial_t^2-\Delta_g$, up to a boundary-fixing isometry. We refer to \cite{IBSPbook} for related results and more references. Microlocal methods have also been used to extract geometric and coefficient information from the singularities measured in waves \cite{GreenleafUhlmann1993, OksanenSaloStefanovUhlmann2024}. Recently, a complete Riemannian metric and lower-order coefficients were recovered, up to natural gauges, from a local source-to-solution map \cite{saksala2025inverse}. This extends the result of \cite{HelinLassasOksanenSaksala2018}, which established almost-sure recovery, up to isometry, from long-time empirical correlations of waves generated by an unknown spacetime white-noise source, assuming the metric is non-trapping and is Euclidean outside a compact set.

In contrast to the works mentioned above, in our setting we have no control over sources, and our data consist of an unlabeled collection of waves generated from point sources in an open set. Compared with Dirichlet-to-Neumann or source-to-solution data, our data are weaker in the sense that we cannot control the sources to produce special solutions of any kind. When the source points range over the whole unknown domain, as in Corollary~\ref{cor:sources_and_unknowns_agree_receivers anywhere}, however, the resulting geometric data contain travel times from every point in an open set.
It was first shown in \cite{IBSPbook} that the travel-time data determine a compact Riemannian manifold with boundary. The stability of this determination was later studied, for instance, in \cite{deHoopIlmavirtaLassasSaksala2023, ilmavirta2024lipschitzstabilitytraveltime, ilmavirta2023three, katsuda2007stability}.
In a similar measurement geometry, it was shown in \cite{ivanov2020distance1, deHoopSaksala2019, LassasSaksala2019} that an unlabeled collection of distance-difference functions associated with interior points determines the underlying Riemannian manifold up to isometry. In particular, our proof of Theorem~\ref{thm:reduction} is closely related to the proof in \cite[Proposition 3.1]{LassasSaksala2019}, which concerns the wave equation with trivial initial conditions and spacetime interior point sources as forcing functions. Related reconstruction problems involving travel-time and scattering data generated from unknown interior sources were studied in \cite{LassasSaksalaZhou2018}. We show that propagation of singularities and the sharp finite speed of propagation property reduce the unlabeled wave data to the corresponding unlabeled travel-time data, without requiring simplicity, convexity, or non-trapping assumptions. In the setting of Corollary~\ref{cor:sources_and_unknowns_agree_receivers anywhere}, these travel-time data determine the wave speed without any such geometric assumptions.

Similar inverse problems have been studied in elastodynamics. For isotropic elastic tensors, Rachele \cite{Racheleboundary2000, Rachelespeeds2000, Racheledensity2003} proved uniqueness of boundary jets of the Lam{\'e} parameters and density, and showed that under additional assumptions on the pressure and shear wave speeds, the Lam\'e parameters as well as the density can be uniquely determined from the Dirichlet-to-Neumann map. She has also studied uniqueness for elastic media with residual stress \cite{Racheleresidualstress2003}. Hansen and Uhlmann \cite{HansenUhlmann2003} recovered the compressional and shear lens relations from boundary measurements for elastic media with bounded residual stress in a setting that allows conjugate points and caustics. Stefanov, Uhlmann, and Vasy \cite{StefanovUhlmannVasy2018} proved local recovery of the compressional and shear speeds under a convex foliation assumption. More recently, Zhai \cite{JianZhai2026} recovered the Lam{\'e} parameters and density under such a convex foliation assumption, while in \cite{IlmavirtaSaksalaYan2026} it was shown that any Lam{\'e} system whose boundary measurements agree with those arising from a homogeneous Lam{\'e} system must itself be homogeneous. These results motivate the study of analogous problems in elastodynamics with interior sources. The scalar case considered here allows us to isolate the geometric reduction step from the complications that arise in elastic systems.

\subsection{Outline of the paper}

Section~\ref{sec:preliminaries} collects some geometric and microlocal preliminaries needed in the main argument, including well-posedness of the forward problem, sharp finite speed of wave propagation, propagation of singularities for real principal type operators, and a theorem that describes the behavior of wave front sets under pullbacks of distributions. In Section~\ref{sec:reduction}, we prove Theorem~\ref{thm:reduction} by identifying the first arrival time of the wave $u^{y,A}$ at a point with the Riemannian distance from the source $y$ to the arrival point. A proof of Proposition~\ref{prop:geom} is given in Section~\ref{sec:proof_of_geom_prop}, while Section~\ref{sec:geometric-recovery} then proves Corollaries~\ref{cor:sources_and_unknowns_agree_receivers anywhere} and \ref{cor:enclosure_sources_enclosure_measurement}. The common thread of the proofs is the reduction of the wave measurements to travel-time data and the solution of the corresponding geometric inverse problems.

Appendix~\ref{app:exterior_metrics} gives examples of geometric conditions under which the assumption \eqref{eq:ext_assumptions_for_ci} of Corollary~\ref{cor:enclosure_sources_enclosure_measurement} holds.
Appendix~\ref{app:postponed-proofs} collects the proofs of some auxiliary lemmas that were omitted from the main text. In Appendix~\ref{app:FIO-proof}, we provide an alternative proof of Proposition~\ref{prop:llc}, which characterizes the wave front set of $u^{y,A}$, using a local Fourier integral parametrix. Although more involved in the scalar setting, this proof provides a natural template for generalizing the argument to elastic systems, where the characteristic set splits into multiple sheets and the simpler argument using pullbacks of wave front sets does not directly apply.

\subsection{Acknowledgments}

JI was supported by the Research Council of Finland (Flagship of Advanced Mathematics for Sensing Imaging and Modelling grant 359208; Centre of Excellence of Inverse Modelling and Imaging grant 353092; and other grants 351665, 358047, 360434) and a Väisälä project grant by the Finnish Academy of Science and Letters.
TS was supported by the National Science Foundation (DMS-2510272) and the Simons Foundation Travel Support for Mathematicians (MPS-TSM-00013291).
AT was supported by the Research Council of Finland (Flagship of Advanced Mathematics for Sensing Imaging and Modelling grant 359208; Centre of Excellence of Inverse Modelling and Imaging grant 353092; and other grants 358047, 360434).
GU was supported by the National Science Foundation.

We are grateful to Katya Krupchyk for helpful discussions and suggestions, and for generously hosting JI, TS, and GU at the University of California, Irvine, where a portion of this work was completed.

\subsection*{AI disclosure}
During the preparation of this manuscript, OpenAI's ChatGPT GPT-6 Astra was used to edit the language, condense the presentation, and assist with locating references and developing and checking selected mathematical arguments. Particularly,
Appendix~\ref{app:exterior_metrics} was developed with assistance from these conversations.
The authors take responsibility for all mathematical statements and proofs presented in this paper.

\section{Preliminaries} \label{sec:preliminaries}

In this section, we collect some geometric, analytic, and microlocal facts used in the reduction from wave measurements to travel-time data.

\subsection{The Lorentzian formulation of the forward problem}
We begin by considering the Cauchy problem \begin{equation}\label{general-ivp}
    \begin{cases}
        (\partial_t^2-c(x)^{\alpha}\divg_x(c(x)^{2-\alpha}\nabla_x))u(t,x) = 0, & (t,x) \in \R\times \R^n, \\
        u(0,x) = f_0(x), & x \in \R^n, \\
        \partial_t u(0,x) = f_1(x), & x \in \R^n,
    \end{cases}
\end{equation}
where $f_0,f_1 \in \mathcal{D}'(\R^n)$. Denote by $P_{c,\alpha}$ the differential operator
\begin{equation}\label{eq:operator}
    P_{c,\alpha}u(t,x) = \partial_t^2u(t,x)-c(x)^{\alpha}\divg_x(c(x)^{2-\alpha}\nabla_x u(t,x)), \qquad (t,x) \in \R\times\R^n.
\end{equation}
Expanding the spatial part, we get
$$P_{c,\alpha}u = \left(\partial_t^2u -c^2\Delta_xu\right) - (2-\alpha)c\nabla_xc \cdot \nabla_x u.$$
So the principal part of $P_{c,\alpha}$ is independent of $\alpha$, and its principal symbol is the following homogeneous polynomial of degree two:
\begin{equation} \label{eq:principalsymb} p(t,x,\tau,\xi) \,
= \, -\tau^2+c(x)^2|\xi|^2 \,
=\, -(\tau +c(x)|\xi|)(\tau - c(x)|\xi|).
\end{equation}

We equip the space $\R_t\times \R^n_x$ with the Lorentzian metric
\begin{equation}\label{eq:lorentzmetric}
   \bar{g} = -dt^2 + g_c = -dt^2+c^{-2}e .
\end{equation}

Let us recall some standard terminology related to Lorentzian manifolds; see \cite[Sec.~1.3]{baer2007wave} or \cite[Ch.~14]{ONeill} for a more detailed treatment. A nonzero tangent vector $v$ in a Lorentzian manifold $(\X, \widetilde{g})$ is said to be \emph{timelike}, \emph{lightlike} (or \emph{null}), or \emph{spacelike} if
$$\widetilde{g}(v,v) < 0, \quad \widetilde{g}(v,v)=0, \quad \text{or} \quad \widetilde{g}(v,v)>0,$$
respectively. A vector $v$ is said to be \emph{causal} if it is either timelike or null, i.e., $\widetilde{g}(v,v)\leq 0$. A \emph{time orientation} on $\X$ is a smooth timelike vector field $Z$. Given such a $Z$, a causal vector $v\in T_p\X$ is called \emph{future-directed} if
$\widetilde g(v,Z_p)<0$, and past-directed if
$\widetilde g(v,Z_p)>0$. We use the time orientation on
$(\R_t\times \R_x^n, \bar{g})$ determined by $Z=\partial_t$.

A piecewise $C^1$ curve $\bar \gamma$ in $\X$ is said to be timelike (resp., lightlike, spacelike, causal, future-directed, or past-directed) if its tangent vector is timelike (resp., lightlike, spacelike, causal, future-directed, or past-directed) at each point of differentiability. The \emph{causal future} $J^+(p)$ (resp., the \emph{causal past} $J^-(p)$) of a point $p\in \X$ is the set of all points $q\in \X$ such that there exists a future-directed (resp., past-directed) causal curve from $p$ to $q$. The causal future and past of a subset $S \subset \X$ are defined by $J^{\pm}(S) := \cup_{p \in S}J^{\pm}(p)$. We also define the \emph{causal set} of $S$ by $J(S) := J^+(S)\cup J^-(S)$.

A hypersurface $\Sigma \subset \X$ is said to be \emph{Cauchy} if every inextendible future- or past-directed timelike curve in $\X$ meets $\Sigma$ exactly once. If $\X$ admits a Cauchy hypersurface, it is said to be \emph{globally hyperbolic} (see \cite[Theorem~1.3.10]{baer2007wave} for equivalent definitions).
We recall the well-known fact that our spacetime $(\R\times\R^n, \bar{g})$ is globally hyperbolic.

\begin{Lemma}\label{lemma:globalhyp}
Let $\bar g$ be as in \eqref{eq:lorentzmetric}.
    Every inextendible past- or future-directed causal curve in the Lorentzian manifold $(\R\times\R^n, \bar{g})$ meets each slice
    $$\Sigma_t = \{t\}\times \R^n \qquad (t \in \R)$$
    exactly once. In particular, $(\R\times \R^n, \bar{g})$ is a globally hyperbolic Lorentzian manifold, and each $\Sigma_t$ is a Cauchy hypersurface.
\end{Lemma}

We give a proof in Appendix~\ref{app:postponed-proofs} for the convenience of the reader.

\subsection{Well-posedness and sharp finite speed of propagation}

A second-order differential operator on a Lorentzian manifold is said to be \emph{normally hyperbolic}
if its principal symbol coincides with the quadratic form induced by the metric on cotangent vectors (see \cite[Sec.~1.5]{baer2007wave}). For $(\R_t\times \R^n_x, \bar{g})$, the quadratic form is
\begin{equation*}
    \bar{g}^{-1}_{(t,x)}((\tau,\xi),(\tau,\xi)) = -\tau^2 + c(x)^2|\xi|_e^2
  \qquad \text{for all } (t,x,\tau,\xi) \in T^*(\R\times \R^n),
\end{equation*}
which coincides with the principal symbol $p(t,x,\tau,\xi)$ of $P_{c,\alpha}$ from \eqref{eq:principalsymb}. This shows that $P_{c,\alpha}$ is indeed a normally hyperbolic operator. There is a standard Cauchy theory for such operators on globally hyperbolic Lorentzian manifolds; see, for instance, \cite[Ch.~3]{baer2007wave} for the case of smooth Cauchy data. For distributional Cauchy data, we will use the formulation in \cite{ForwardProblemPaper}.

\begin{Lemma}[Well-posedness]\label{well-posedness}
Let $c \in C^\infty(\R^n)$ satisfy \eqref{eq:uniform_sound_speed}. Given $f_0,f_1 \in \mathcal{D}'(\R^n)$, there exists a distributional solution $u(t,x)\in \mathcal{D}'(\R\times \R^n)$ of \eqref{general-ivp} with the following approximation property: if $f_{0,j}, f_{1,j} \in C_c^\infty(\R^n)$ are such that
$$f_{0,j} \longrightarrow f_0, \qquad f_{1,j} \longrightarrow f_1 \quad \text{in } \mathcal{D}'(\R^n)$$
as $j \to \infty$, and $u_j\in C^\infty(\R\times\R^n)$ is the classical solution of \eqref{general-ivp} corresponding to the Cauchy data $(f_{0,j}, f_{1,j})$, then
$$u_j \longrightarrow u \quad \text{in } \mathcal{D}'(\R\times \R^n).$$
The solution $u$ is unique among the distributional solutions satisfying this approximation property. Moreover, the map
$$(f_0, f_1) \mapsto u$$
is sequentially continuous from $\mathcal{D}'(\R^n)\times \mathcal{D}'(\R^n)$ to $\mathcal{D}'(\R\times\R^n)$.
\end{Lemma}
\begin{proof}
    This will follow from \cite[Theorem~2.2]{ForwardProblemPaper} as soon as we verify its hypotheses. We saw in Lemma~\ref{lemma:globalhyp} that $(\R\times\R^n,\bar{g})$ is a globally hyperbolic Lorentzian manifold, and each slice $\Sigma_t =\{t\}\times \R^n$ is a Cauchy hypersurface. We have also verified that $P_{c,\alpha}$ is a normally hyperbolic operator on $(\R\times \R^n, \bar g)$. Next, we verify that $P_{c,\alpha}$ is \emph{formally self-adjoint} in the sense of satisfying \cite[Eq.~(2.2)]{ForwardProblemPaper}. Indeed, integration by parts with the weight $h_E = c(x)^{n-\alpha}$ shows that
    \begin{equation*}
     \int_{\R\times \R^n} h_E \psi P_{c,\alpha} \varphi\, dV_{\bar{g}} = \int_{\R\times \R^n}h_E\varphi P_{c,\alpha} \psi\, dV_{\bar{g}} \qquad \text{for all }\varphi,\psi \in C_c^\infty(\R\times \R^n).
    \end{equation*}
    So we may apply \cite[Theorem~2.2]{ForwardProblemPaper} to the operator $P_{c,\alpha}$ on the Lorentzian manifold $(\R\times \R^n, \bar{g})$, the Cauchy hypersurface $\Sigma_0$, and the Cauchy data $u_0,u_1 \in \mathcal{D}'(\Sigma_0)$ given by
    $$u_0(0,x) = f_0(x), \qquad u_1(0,x) = f_1(x), \qquad (x\in \R^n),$$
    to obtain a distributional solution $u \in\mathcal{D}'(\R\times \R^n)$ to \eqref{general-ivp} that satisfies the stated approximation property. The solution map $(u_0,u_1)\mapsto u$ in that theorem satisfies a stronger sequential continuity property than ordinary distributional convergence, so it immediately implies the sequential continuity of $(f_0, f_1)\mapsto u$ from $\mathcal{D}'(\R^n)\times \mathcal{D}'(\R^n)$ to $\mathcal{D}'(\R\times \R^n)$.

    To prove uniqueness, suppose $v$ is another distributional solution of \eqref{general-ivp} with the stated approximation property. Choose $f_{0,j},f_{1,j} \in C_c^\infty(\R^n)$ that converge to $f_0, f_1$ in $\mathcal{D}'(\R^n)$, and let $u_j\in C^\infty(\R\times \R^n)$ be the corresponding classical solutions. Then the approximation property implies
    $$u_j \longrightarrow u, \qquad \text{and} \qquad u_j \longrightarrow v \qquad \text{in }\mathcal{D}'(\R\times \R^n).$$
    Since distributional limits are unique, we conclude that $u=v$.
\end{proof}

\begin{remark}
    Henceforth, whenever we refer to the solution of the Cauchy problem \eqref{general-ivp} or its special cases, we will mean the unique distributional solution satisfying the approximation property stated in Lemma~\ref{well-posedness}.
\end{remark}

We next recall the sharp finite speed of propagation property. Geometrically, this states that if the Cauchy data $f_0, f_1$ in \eqref{general-ivp} are supported in a compact set $K\subset \R^n$, the corresponding solution $u$ is supported in the causal set $J(\{0\}\times K)$. For smooth Cauchy data, this is a standard result for normally hyperbolic operators on globally hyperbolic manifolds; see, for example, \cite[Theorem~3.2.11]{baer2007wave}. In the following lemma, we record the corresponding statement for distributional Cauchy data, in the form needed later.

\begin{Lemma}[Sharp finite speed of wave propagation]\label{sharp}
    Let $f_0, f_1 \in \mathcal{E}'(\R^n)$ and $u$ be the solution of \eqref{general-ivp}. Then
    \begin{equation}
         \supp(u) \subset \left\{ (t,x) \in \R\times \R^n \ : \ d_c(x, \supp(f_0)\cup \supp(f_1)) \leq |t|\right\}.
    \end{equation}
    In particular, for the point source solution $u^{y,A}$ of \eqref{ivp}, we have
    \begin{equation}
        u^{y,A} = 0 \qquad \text{near } (t,x) \text{ whenever } |t| < d_c(x,y).
    \end{equation}
\end{Lemma}
\begin{proof}
Let $K=\supp(f_0)\cup\supp(f_1)$. Since $f_0$ and $f_1$ are compactly supported, we may choose $f_{0,j},f_{1,j}\in C_c^\infty(\mathbb R^n)$, $(j \in \N)$, such that $f_{0,j}\to f_0$, $f_{1,j}\to f_1$ in $\mathcal{D}'(\R^n)$, and such that
$$\supp(f_{0,j}) \cup \supp(f_{1,j}) \subset U_j, \qquad \text{where } U_j = \{x \in \R^n \, : d_c(x, K) < 1/j \}.$$
Let \(u_j\) be the solutions of \eqref{general-ivp} corresponding to the Cauchy data $f_{0,j}, f_{1,j}$. By the sharp finite speed of propagation property for smooth Cauchy data (see, for instance, \cite[Theorem~3.2.11]{baer2007wave}), we have
\begin{equation}\label{eq:suppujinJUj}
    \supp(u_j) \subset J(\{0\}\times U_j).
\end{equation}
Suppose $(t,x) \in J(\{0\}\times U_j)$. Then there exists $x_0 \in U_j$ and a piecewise $C^1$ causal curve from $(0,x_0)$ to $(t,x)$. After reparametrizing, we may assume that the curve is of the form
$$
\bar \gamma(s) = \left(\sgn(t)s, x(s)\right), \qquad 0\leq s \leq |t|.
$$
The causality condition implies that $|\dot{x}(s)|_{g_c} \leq 1$ at every point of differentiability, and therefore,
$$d_c(x,U_j) \leq d_c(x_0,x) \leq \int_0^{|t|} |\dot{x}(s)|_{g_c}\, ds \leq |t|.$$
Combining this with \eqref{eq:suppujinJUj}, we get
\begin{equation*}
    \supp(u_j) \subset J(\{0\}\times U_j) \subset \{(t,x) \in \R \times \R^n: d_c(x,U_j) \leq |t|\}.
\end{equation*}
By the sequential continuity of the solution map with respect to the Cauchy data established in Lemma~\ref{well-posedness}, we have $u_j\to u$ in $\mathcal D'(\mathbb R\times\mathbb R^n)$. Choose a test function $\varphi \in C_c^\infty(\R\times \R^n)$ which is supported in $\{(t,x)\in \R\times \R^n \, : \, d_c(x,K) > |t|\}$. By compactness of $\supp(\varphi)$, we must have
$$\supp(\varphi) \subset \{(t,x) \in \R\times \R^n \, : \, d_c(x,K) > |t|+1/j\} \subset \{(t,x)\in \R\times \R^n \, : \, d_c(x,U_j) > |t|\}$$
for all $j$ sufficiently large. Consequently,
$$\langle u, \varphi\rangle = \lim_{j\to \infty}\langle u_j, \varphi\rangle = 0.$$
This proves that $u$ must be supported in $\{(t,x)\in \R\times \R^n \, : \, d_c(x,K) \leq |t|\}$ as claimed.
\end{proof}

\subsection{Wave front sets and propagation of singularities} \label{subsec:propagation-preliminaries}

We next recall the microlocal analysis results used to detect the first arrival of the wave $u^{y,A}$ initiated by a point source at $y$. Let $X\subset \R^m$ be open and let $Q$ be a classical pseudodifferential operator on $X$ with principal symbol $q$. The Hamiltonian vector field of $q$ on $T^*X$ is
\begin{equation}\label{eq:Hamiltonian-vector-field}
    H_q := \sum_{j=1}^m \left(\frac{\partial q}{\partial \xi_j}\frac{\partial}{\partial x_j} - \frac{\partial q}{\partial x_j}\frac{\partial}{\partial \xi_j}\right).
\end{equation}
The integral curves of $H_q$ are called the bicharacteristics of $q$. If the curve is contained in $\{q=0\}$, we call it a null bicharacteristic.

\begin{Definition}
    We say that $Q$ is of real principal type if its principal symbol $q(x,\xi)$ is real-valued and the differential $d_{x,\xi}q(x,\xi)$ is not a scalar multiple of $\sum_j \xi_j dx_j$ on $T^*X \setminus 0$.
\end{Definition}

In fact, it follows from the homogeneity of $q$ that $d_\xi q(x,\xi) \neq 0$ when $q(x,\xi)\neq 0$. Hence, $dq(x,\xi)$ can only be a scalar multiple of $\sum_j \xi_j dx_j$ on the characteristic set
    \begin{equation}
    \Char(Q) = \{(x,\xi)\in T^*X \setminus 0 \ | \ q(x,\xi)=0\}.
    \end{equation}

For the hyperbolic operator $P_{c,\alpha}$ from \eqref{eq:operator}, the principal symbol $p(t,x,\tau,\xi)=-\tau^2+c(x)^2|\xi|^2$ is clearly real-valued. Moreover, on the characteristic set $\Char(P_{c,\alpha})$, we have seen that $\tau = \pm c(x)|\xi| \neq 0$.
Hence
\begin{equation*}\partial_\tau p = -2\tau \neq 0\end{equation*}
on $\Char(P_{c,\alpha})$, which implies that $P_{c,\alpha}$ is of real principal type.

\begin{Lemma}[Propagation of singularities, {\cite[Theorem~8.1]{Gr-Sj}}]\label{propa}
  Let $Q$ be a properly supported pseudodifferential operator of real principal type on an open set $X\subset \R^n$, with principal symbol $q$. Suppose that $u \in \mathcal{D}'(X)$ satisfies $Qu=0$. Then
  $$WF(u) \subset \Char(Q),$$
  and each maximally extended null bicharacteristic curve of $q$ is either completely contained in $WF(u)$ or disjoint from it. In other words,
  $WF(u)$ is a union of maximally extended null bicharacteristics of $q$.

\end{Lemma}

We will also use the following result on pullbacks of distributions. Let $X\subset \R^m$ and $Y\subset \R^n$ be open sets. For a smooth map $F:X\to Y$, we define the set of normals by
$$N_F = \{(F(x),\eta) \in T^*Y \, : \, (dF_x)^t\eta = 0\}.$$
If $\Gamma \subset T^*Y$, define
$$F^*\Gamma = \{(x, (dF_x)^t\eta) \in T^*X \, : \, (F(x),\eta) \in \Gamma\}.$$

\begin{Lemma}[Pullbacks of wave front sets, {\cite[Theorem~8.2.4]{Horm1}}]\label{lemma:pullback-wavefront}
    Let $F:X \to Y$ be smooth and $u \in \mathcal{D}'(Y)$. If
    $$N_F \cap WF(u) = \emptyset,$$
    then the pullback $F^*u \in \mathcal{D}'(X)$ is well defined, agrees with $u\circ F$ whenever $u$ is a function, and satisfies
    \begin{equation} \label{eq:wavefrontset-inclusion}
        WF(F^*u) \subset F^*WF(u).
    \end{equation}
\end{Lemma}

This result will be applied in Section~\ref{sec:reduction} to the distribution $u^{y,A}$ and the inclusion map $\iota$ of the initial surface $\{t=0\}$ into $\R\times\R^n$.

\section{Reduction to Travel-Time Data} \label{sec:reduction}

The goal of this section is to prove Theorem~\ref{thm:reduction}. We begin by introducing the first arrival time function.
\begin{Definition}
    For $y,z \in \R^n$, define the first arrival time at $z$ from $y$ by
    \begin{equation}\label{eq:first-arrival-time}
    T_c(y,z) = \sup \left( \{0\}\cup\left\{ t >0\ : \ u^{y,1}|_{(0,t)\times V} = 0 \text{ for some open neighborhood } V \subset \R^n \text{ of }z \right\}\right).
\end{equation}
\end{Definition}

Clearly, $T_c(y,z)$ is well defined in $[0,+\infty]$. If $T> M\diam_e(\source \cup \rec)$, then $T_c(y,z)$ is determined by the measurement data $D_c^{\source,\rec,T}$ for all $y \in \source$ and $z\in \rec$. Moreover, by linearity, $u^{y,A}=Au^{y,1}$, so replacing $u^{y,1}$ in \eqref{eq:first-arrival-time} by $u^{y,A}$ for any $A\neq 0$ does not change the definition of the first arrival time function.

We now use results on propagation of singularities and pullbacks of distributions from Subsection~\ref{subsec:propagation-preliminaries} to give a complete description of the wave front set of $u^{y,A}$. Consider the Lorentzian metric
$\bar{g}$ on $\R\times\R^n$ defined in \eqref{eq:lorentzmetric}. Recall that, given a point $(t_0,x_0) \in \R\times \R^n$, the \emph{light cone} at $(t_0,x_0)$ is defined as the union of all maximal lightlike (or null) geodesics passing through $(t_0,x_0)$. We now define a \emph{lifted} version of this object in the cotangent bundle.

\begin{Definition}
    Let $(t_0,x_0) \in \R\times \R^n$ and $p$ be as in \eqref{eq:principalsymb}. The lifted light cone based at $(t_0,x_0)$, denoted by $LLC_{g_c}(t_0,x_0)$, is the union of all null bicharacteristics of $p$ in $T^*(\R\times \R^n)\setminus 0$ that intersect $T^*_{(t_0,x_0)}(\R\times\R^n)$.

\end{Definition}

The next lemma implies that the lifted light cone based at $(t_0,x_0)$ is the union of lifts to $T^\ast (\R\times \R^n)\setminus 0$ of all maximal null geodesics of the Lorentzian metric $\bar g$ passing through $(t_0,x_0)$.

\begin{Lemma}
\label{lemma:projection_to_null_geos}
Let $\bar g$ be as in \eqref{eq:lorentzmetric}, and let $p$, given by \eqref{eq:principalsymb}, be the principal symbol of the wave operator $\p_t^2-\Delta_{g_c}$.
The base projection of a null bicharacteristic
of $p$ is a null geodesic
of $\overline{g}$.
\end{Lemma}

The result is well known, but we give a proof in Appendix~\ref{app:postponed-proofs} for the convenience of the reader.

\begin{Proposition}\label{prop:llc}
    For each $A\in \R\setminus\{0\}$ and $y\in \R^n$, the wave front set of the solution $u^{y,A}$ of \eqref{ivp} is the lifted light cone $LLC_{g_c}(0,y)$.
\end{Proposition}
\begin{proof}
    Assume without loss of generality that $A=1$, and write $u^y = u^{y,1}$. We first show that all characteristic covectors at $(0,y)$ belong to $WF(u^y)$. That is,
    \begin{equation}
        \label{eq:Char_set_at_0_y}
    \{(0,y,\tau,\xi) \in T^*(\R\times \R^n) \setminus 0 \, : \, \tau = \pm c(y)|\xi|\}
    =
    LLC_{g_c}(0,y) \cap T^*_{(0,y)}(\R\times\R^n)
    \subset
    WF(u^y).
    \end{equation}

    Let $\iota:\R^n \to \R\times \R^n$ be the inclusion $\iota(x) = (0,x)$. By Lemma~\ref{lemma:pullback-wavefront}, the pullback $\iota^*u^y \in \mathcal{D}'(\R^n)$ is well defined provided the wave front set of $u^y$ is disjoint from
    \begin{equation*}
        N_\iota := \left\{(\iota(x), \zeta) \in T^*(\R\times\R^n)\, :\, x \in \R^n, \, (d\iota_x)^t(\zeta) = 0 \right\}.
    \end{equation*}
    Since $d\iota_x(v) = (0,v)$, its transpose satisfies $(d\iota_x)^t(\tau,\xi) = \xi$. Therefore,
    \begin{equation*}
        N_\iota = \left\{(0,x,\tau,0) \, : x \in \R^n, \, \tau \in \R \right\}.
    \end{equation*}
    Since $P_{c,\alpha}u^y=0$, Lemma~\ref{propa} gives
    \begin{equation}\label{eq:WFsubsetChar}
        WF(u^y) \subset \Char(P_{c,\alpha}):= \{(t,x,\tau,\xi) \in T^*(\R\times \R^n) \setminus 0 \, : \, \tau = \pm c(x)|\xi|\},
    \end{equation}
    which is clearly disjoint from $N_\iota$. So $\iota^*u^y$ is well defined in $\mathcal{D}'(\R^n)$.
    Moreover, by the wave front set inclusion \eqref{eq:wavefrontset-inclusion} given in Lemma~\ref{lemma:pullback-wavefront},
    \begin{equation}\label{eq:WFpullback}
        WF(\iota^* u^y) \subset \iota^*(WF(u^y)) =\{(x,\eta)\, : \, (0,x,\tau,\eta) \in WF(u^y) \text{ for some }\tau \in \R\}.
    \end{equation}
    The initial condition $u^y(0,\dummy) = \delta_y$ gives $\iota^*(u^y) = u^y(0,\dummy)=\delta_y$. So we have
    $$WF(\iota^*u^y) = \{(y,\eta) \in T^*_y\R^n\, : \, \eta \neq 0\}.$$
    Therefore, by \eqref{eq:WFpullback}, for each $\eta \in T^*_y\R^n\setminus 0$, there exists $\tau \in \R$ such that
    $(0,y,\tau,\eta) \in WF(u^y),$
    and by \eqref{eq:WFsubsetChar}, $\tau = -c(y)|\eta|$ or $\tau = +c(y)|\eta|$.

    We now show that both of these covectors are in $WF(u^y)$. Indeed, it is easy to verify that the distribution
    $$v(t,x) = u^y(-t,x)$$
    solves the same Cauchy problem \eqref{ivp} as $u^y$ and satisfies the same approximation property stated in Lemma~\ref{well-posedness}.
    So by uniqueness, $v= u^y$. Consequently, $WF(u^y)$ is invariant under the map
    $$
    (t,x,\tau,\xi) \mapsto (-t,x,-\tau, \xi).
    $$
    Hence, for every $\eta\in T^*_y\R^n\setminus 0$, both characteristic covectors $(0,y,\pm c(y)|\eta|,\eta)$ are in $WF(u^y)$. This proves the inclusion \eqref{eq:Char_set_at_0_y}. The propagation of singularities result, Lemma~\ref{propa}, now implies that every null bicharacteristic in $\R\times \R^n$ that passes over $(0,y)$ is contained in $WF(u^y)$. Therefore,
    \begin{equation}\label{eq:reverse-inc2}
        LLC_{g_c}(0,y) \subset WF(u^y).
    \end{equation}

    We now prove the reverse inclusion. Suppose $\rho =(t,x,\tau,\xi)\in WF(u^y)$. By \eqref{eq:WFsubsetChar}, $\rho$ is a characteristic covector, i.e., $\tau = \pm c(x)|\xi|$. Let
    $$\Gamma(s) = (t(s),x(s),\tau(s),\xi(s))$$
    be the maximally extended null bicharacteristic of $p$ through $\rho$. By Lemma~\ref{lemma:projection_to_null_geos}, its projection $(t(s),x(s))$ is a maximally extended null geodesic of $\bar g$. So by Lemma~\ref{lemma:globalhyp}, there exists $s_0\in \R$ such that $t(s_0)=0$. Applying Lemma~\ref{propa} again, we must have
    \begin{equation}\label{eq:WF_at_0}
        (0,x(s_0),\tau(s_0),\xi(s_0)) \in WF(u^y).
    \end{equation}
    If $x(s_0) \neq y$, then Lemma~\ref{sharp} implies that $u^y$ vanishes near $(0,x(s_0))$. This contradicts \eqref{eq:WF_at_0}, and we conclude that $x(s_0)=y$. In particular, $\Gamma$ passes over $(0,y)$, and therefore, $\rho \in LLC_{g_c}(0,y)$. Thus, we have proved that
    \begin{equation}\label{eq:inc2}
        WF(u^y) \subset LLC_{g_c}(0,y).
    \end{equation}
    This, together with \eqref{eq:reverse-inc2}, proves the claim.
\end{proof}

We can now show that the first arrival time function $T_c(y,\dummy)$ coincides with the Riemannian distance function $d_{c}(y,\dummy)$, thus reducing the problem of recovering $c$ to a geometric inverse problem.

\begin{Proposition}\label{prop:first-arrival}
    Let $c \in C^\infty(\R^n)$ satisfy \eqref{eq:uniform_sound_speed}, and let $g_c$ be the conformally Euclidean metric in \eqref{eq:conf_euc_metric}. Then for every $y,z\in \R^n$,
    \begin{equation}\label{eq:Teqd}
        T_c(y,z) = d_c(y,z).
    \end{equation}
\end{Proposition}
\begin{proof}
Let us first consider the case $y=z$. By Proposition~\ref{prop:llc}, $u^{y,1}$ is singular, and hence nonvanishing, in every neighborhood of $(0,y)$. So it follows from the definition that
$$T_c(y,y)= 0 = d_c(y,y).$$
Now suppose $y\neq z$. By Lemma~\ref{sharp}, the solution $u^{y,1}$ of \eqref{ivp} vanishes in a neighborhood of $(t,z)$ whenever $0 < t < d_{c}(y,z)$. Hence it follows directly from the definition of the first arrival time that
\begin{equation}\label{eq:Tgeqd}
        T_{c}(y,z) \geq d_{c}(y,z).
\end{equation}

We now prove the reverse inequality. By the assumption \eqref{eq:uniform_sound_speed}, the metric $g_c$ is complete, and the Hopf--Rinow theorem gives a minimizing unit-speed $g_c$-geodesic $\gamma:[0,\ell]\to \R^n$ such that $\gamma(0)=y$ and $\gamma(\ell)=z$. Here $\ell=d_{c}(y,z)$. Thus, the spacetime curve
$
\bar \gamma(s) = (s, \gamma(s)), \: 0\leq s \leq \ell,
$
is a null geodesic for the Lorentzian metric $\bar{g}$ on $\R\times \R^n$. Moreover, it follows from Lemma~\ref{lemma:projection_to_null_geos} that $\bar \gamma$ lifts to an integral curve $\Gamma$ of $H_p$ on $T^*(\R \times \R^n)$. In particular, there is a nonzero covector $(\ell,z,\tau(\ell),\xi(\ell)) = \Gamma(\ell)\in T^*_{(\ell,z)}(\R\times \R^n)$ such that
$$(\ell,z,\tau(\ell),\xi(\ell)) \in LLC_{g_c}(0,y).$$
By Proposition~\ref{prop:llc}, $LLC_{g_c}(0,y) = WF(u^{y,1})$. Therefore, $u^{y,1}$ is not smooth in any neighborhood of $(\ell,z)$ and cannot vanish in any such neighborhood. So it follows from the definition of the first arrival time function that
\begin{equation}\label{Tleqd}
    T_c(y,z) \leq \ell = d_c(y,z).
\end{equation}
Combining the inequalities \eqref{eq:Tgeqd} and \eqref{Tleqd} proves the result.
\end{proof}

\begin{proof}[Proof of Theorem~\ref{thm:reduction}]
For $j\in \{1,2\}$, let $u_j^{y,A}$ denote the solution of \eqref{ivp} for $c=c_j$. We first observe that every relevant arrival occurs before time $T$. Indeed, for all $y \in \source$, $z\in \rec$, and $j \in \{1,2\}$, we have
$$d_{c_j}(y,z) \leq Md_e(y,z) \leq M\diam_e(\source\cup\rec) < T.$$
By Proposition~\ref{prop:first-arrival}, we have
$$T_{c_j}(y,z) = d_{c_j}(y,z) < T $$
for all $y \in \source$ and $z \in \rec$. Thus, $T_{c_j}$ is completely determined by the restriction of $u_j^{y,A}$ to $(0,T)\times \rec$.

We now prove \eqref{eq:TTT_agree}. Let $y_1 \in \source$ and choose any $A_1 \neq 0$. By the assumption \eqref{eq:wave_data}, there exist $y_2 \in \source$ and $A_2 \neq 0$ such that
\[
u_1^{y_1,A_1}|_{(0,T)\times \rec}
=
u_2^{y_2,A_2}|_{(0,T)\times \rec}.
\]
This implies by definition that
\[
T_{c_1}(y_1,\dummy)\big|_{\rec}
=
T_{c_2}(y_2,\dummy)\big|_{\rec},
\]
and Proposition~\ref{prop:first-arrival} yields
\[
d_{c_1}(y_1,\dummy)\big|_{\rec}
=
d_{c_2}(y_2,\dummy)\big|_{\rec} \in \TTD{c_2}{\source,\rec}(\source).
\]
Consequently,
$$\TTD{c_1}{\source,\rec}(\source) \subset \TTD{c_2}{\source,\rec}(\source).$$
We get the reverse inclusion by interchanging the roles of $c_1$ and $c_2$, thus proving the equality \eqref{eq:TTT_agree}.
\end{proof}

\section{Proof of Proposition~\ref{prop:geom}}
\label{sec:proof_of_geom_prop}
This section is devoted to the proof of Proposition~\ref{prop:geom}. We follow the strategy introduced in \cite{HelinLassasOksanenSaksala2018, saksala2025inverse}, recalling the main steps of the proof and omitting the details.

Let $X$ be a smooth manifold without boundary equipped with two complete Riemannian metrics $g_1$ and $g_2$. Let $\source,\rec \subset X$ be open sets with smooth boundaries. We assume that both of these sets are bounded with respect to the metric $g_1$. Since closed and bounded sets of a complete Riemannian manifold are compact, and vice versa, the sets $\source$ and $\rec$ are also bounded with respect to the metric $g_2$.

\subsection{Construction of a distance-preserving homeomorphism}
We write $g\in \{g_1,g_2\}$ and recall the definition of the travel-time function
\[
\begin{split}
    &\TTD{g}{\source, \rec}\colon \source \to C(\rec),
\qquad
\TTD{g}{\source, \rec}(y)=d_{g}(y,\dummy)\big|_{\rec}.
\end{split}
\]

Throughout this section, we assume that the travel-time data of $g_1$ and $g_2$ coincide. In other words, we have
\[
\TTD{g_1}{\source, \rec}(\source)
=
\TTD{g_2}{\source, \rec}(\source)
\]
which is just a restatement of the equation \eqref{eq:geom_data}. Hence, continuity of the distance function gives
\begin{equation}
    \label{eq:TTT_revisited_2}
\TTD{g_1}{\source, \bar \rec}(\source)=\TTD{g_2}{\source, \bar \rec}(\source).
\end{equation}

To extend this equation to $\bar \source$, we take $p^1 \in \bar \source$ and a sequence $(p^1_k)_{k=1}^\infty \subset \source$ that converges to $p^1$ with respect to the metric $g_1$. By \eqref{eq:TTT_revisited_2}, we can choose a sequence $(p^2_k)_{k=1}^\infty \subset \source$ such that
\[
d_{g_1}(p^1_k,q)
=
d_{g_2}(p^2_k,q),
\qquad
\text{ for all }
q \in \bar \rec.
\]
Since the set $\bar \source$ is compact, we may without loss of generality assume that $(p^2_k)_{k=1}^\infty$ converges to some $p^2 \in \bar \source$ with respect to the metric $g_2$. Thus, continuity of the distance functions gives
\[
d_{g_1}(p^1,q)
=
d_{g_2}(p^2,q),
\qquad
\text{ for all }
q \in \bar \rec.
\]
Interchanging the roles of $g_1$ and $g_2$, we obtain
\begin{equation}
    \label{eq:TTT_revisited}
\TTD{g_1}{\bar \source, \bar \rec}(\bar \source)=\TTD{g_2}{\bar \source, \bar \rec}(\bar\source)
\subset
C(\bar \rec).
\end{equation}

Since $\bar \rec$ is compact, the pair $(C(\bar \rec),\|\cdot\|_{\infty})$ is a Banach space. Moreover, from the openness of $\rec$ we see that the travel-time maps $\TTD{g}{\bar \source, \bar \rec}$ are topological embeddings, i.e., homeomorphisms onto their images. See, for instance,
\cite[Proposition 3.10]{saksala2025inverse} for more details. Hence, we can use \eqref{eq:TTT_revisited} to see that the map
\[
\varphi\colon \bar \source \to \bar \source,
\qquad
\varphi:=\left(\TTD{g_2}{\bar \source, \bar \rec}\right)^{-1} \circ \TTD{g_1}{\bar \source, \bar \rec}
\]
is a homeomorphism that satisfies
\begin{equation}
    \label{eq:dists_agree}
d_{g_1}(p,q)
=
d_{g_2}(\varphi(p),q),
\qquad
\text{for all }
p \in \bar \source, \: q \in \bar \rec.
\end{equation}
We refer to \cite[Proposition 3.11]{saksala2025inverse} for details of this argument.

\subsection{Equality of the metrics in the receiver set}
\label{subsec:equ_of_metrics_in_R}

Let $q_0 \in \rec$ and choose $p_0 \in \source$ such that $q\mapsto d_{g_1}(p_0,q)$ is smooth near $q_0$. This can be done since the $g_1$-cut locus of $q_0$ is a closed set of measure zero \cite[Theorem 10.34]{LeeRM}, $\source$ is open, and the distance function is smooth outside the cut locus. Hence, equation \eqref{eq:dists_agree} shows that the differentials of the distance functions agree:
\[
d_qd_{g_1}(p,q)|_{q=q_0}
=
d_qd_{g_2}(\varphi(p),q)|_{q=q_0},
\]
for all $p \in \source$ near $p_0$.
Therefore,
\[
S^\ast_{g_1,q_0}X\cap S^\ast_{g_2,q_0}X,
\qquad
\text{where }
S^\ast_{g,q_0}X=\{\xi \in T^\ast_{q_0}X: \: |\xi|_g=1\},
\]
contains a nonempty open set.
Since a Riemannian norm on each fiber is given by an inner product, we have proved that $g^{-1}_1(q_0)=g^{-1}_2(q_0)$ and therefore $g_1(q_0)=g_2(q_0)$. Since $q_0 \in \rec$ was arbitrary, we have proved the equality $g_1=g_2$ in $\bar \rec$ by continuity.

\subsection{Recovery of the smooth structure}
It follows from the proof of \cite[Proposition 3.13]{saksala2025inverse} that the map $\varphi$ is a diffeomorphism of $\source $. Moreover, for each $q \in \rec$, the diffeomorphism $\varphi$ has a local representation
$
\exp_{q,2}\circ \exp_{q,1}^{-1}
$
in the set $\source \setminus C_{g_1}(q)$. Here $C_{g_1}(q)$ stands for the $g_1$-cut locus of $q$.

\subsection{Recovery of the geometry in the source set}
To show that the map $\varphi$ is a Riemannian isometry, we take $p \in \source$ and use the method introduced in Subsection~\ref{subsec:equ_of_metrics_in_R} with equation \eqref{eq:TTT_revisited}, together with the fact that $\varphi$ is a diffeomorphism of $\source$, to show that
$
g_1(p)
=
\varphi^\ast g_2(p).
$
We refer to the proof of \cite[Proposition 3.14]{saksala2025inverse} for more details. Hence, we have proved
$
g_1
=
\varphi^\ast g_2,
$
in $\source$.
The proof of Proposition~\ref{prop:geom} is complete.

\section{Recovery of the Wave Speed from Travel-Time Data} \label{sec:geometric-recovery}

Before proving Corollaries~\ref{cor:sources_and_unknowns_agree_receivers anywhere} and \ref{cor:enclosure_sources_enclosure_measurement} in the subsections below, we record a lemma on the rigidity of conformal maps.
The result is not new, but we could not find it stated as such in the literature, so we give a proof.

\begin{Lemma}[Conformal maps fixing boundary points]
\label{lma:boundary-fixing-conformal}
Let $\Omega\subset\R^n$, $n\geq2$, be a smooth connected open set and $f\colon\bar\Omega\to\bar\Omega$ a continuous
map that is smooth and conformal in the interior.
If $f(x)=x$ for all $x\in\Gamma$, where $\Gamma\subset\partial\Omega$ is relatively open and nonempty, then $f$ is the identity.
\end{Lemma}
\begin{proof}
When $n=2$, we may regard $\Omega$ as a subset of $\C$.
As a conformal map, $f$ is either holomorphic or antiholomorphic on $\Omega$.
If $f$ were antiholomorphic, it would have to reverse the local orientation of the boundary, which is incompatible with fixing all points on $\Gamma$.
Therefore $f$ is holomorphic, as is $h(z)=f(z)-z$.
Fix any $p\in\Gamma$ and let $U$ be a neighborhood of $p$ in $\C$ such that $U\cap\partial\Omega\subset\Gamma$.
If we extend $h$ by zero to a map $\hat h\colon\C\to\C$, then $\hat h$ is continuous in $U$.
The function $\hat h$ is holomorphic in $\Omega\cap U$ and vanishes in $U\setminus\Omega$.
By Rad\'o's theorem, $\hat h$ is holomorphic in all of $U$.
Since $\hat h$ vanishes in the open subset $U\setminus\bar\Omega$, it vanishes identically in $U$.
Now the original $h|_\Omega$ is holomorphic and vanishes in the open set $\Omega\cap U$, so $h=0$ as claimed.
(If $\Gamma=\partial\Omega$, one can simply use the maximum principle to make $h$ vanish.)

When $n\geq3$, Liouville's theorem forces $f$ to be a M\"obius transformation. As such, $f$ is smooth up to the boundary.
Fix any point $p\in\Gamma$.
As $f$ is a smooth map with $f(p)=p$, the differential $Df(p)\colon T_p\R^n\to T_p\R^n$ is a composition of a rotation, a scaling, and possibly a reflection.
This differential is the identity on $T_p\Gamma\subset T_p\R^n$ because $f$ fixes all points on $\Gamma$, so $Df(p)$ is either the identity map or a reflection across $T_p\Gamma$.
The latter is ruled out by $f(\Omega)\subset\bar\Omega$, so indeed $Df(p)=I$.

A general M\"obius transformation $f$ of $\overline{\R^n}$ (the one-point compactification of Euclidean space) is of the form
\begin{equation}
f(x) = b + \lambda\abs{x-a}^{-\eps} Q(x-a)
\end{equation}
with $a,b\in\R^n$, $\lambda>0$, $Q\in O(n)$, and $\eps\in\{0,2\}$.
The derivative is
\begin{equation}
Df(x) = \lambda\abs{x-a}^{-\eps} QR_{x-a}
\end{equation}
with $R_{y}=I-\eps\abs{y}^{-2}yy^T$.
For $f$ to be continuous on $\bar\Omega$, we have $\eps=0$ or $a\notin\bar\Omega$.

Suppose $\eps=2$.
The matrix $R_y$ is the reflection across the hyperplane orthogonal to $y$, so $R_{y}^2=I$ and $R_{y}\in O(n)$.
As $Df(p)$ is a multiple of the identity and $Q\in O(n)$, we find that $Q=R_{p-a}$.
This holds for all $p\in\Gamma$, so the vector $\eta(p)=\frac{p-a}{\abs{p-a}}$ is (locally) constant on $\Gamma$.
The level sets of $\eta$ are one-dimensional and $\dim(\Gamma)=n-1>1$, which leads to a contradiction.
Therefore $\eps=0$.

We are then left with $f(x)=b+\lambda Q(x-a)$.
With the constraints $f(p)=p$ and $Df(p)=I$ at a point $p\in\Gamma$, the only remaining option is indeed $f(x)=x$.
\end{proof}

\subsection{Proof of Corollary~\ref{cor:sources_and_unknowns_agree_receivers anywhere}}
Since $T>M\diam_e(\source \cup \rec)$, Theorem~\ref{thm:reduction} gives \eqref{eq:TTT_agree}.
Hence, Proposition~\ref{prop:geom} implies that $g_{c_1}=g_{c_2}$ in $\rec$. Moreover, there is a homeomorphism $\varphi \colon \bar \source \to \bar \source$ that is a diffeomorphism of $\source$ with the property $g_{c_1}=\varphi^\ast g_{c_2}$.
Since $g_{c_i}=c_i^{-2}e$, we obtain immediately from the assumption $\O=\rec$ that $c_1\equiv c_2$. Hence,
from now on we assume that $\O=\source$.

In Euclidean coordinates, the equation $g_{c_1}=\varphi^\ast g_{c_2}$ implies
\begin{equation}\label{eq:conformal}
 (D\varphi(x))^{\mathsf T}D\varphi(x)
 =\frac{c_2(\varphi(x))^2}{c_1(x)^2}e.
\end{equation}
Therefore, $\varphi$ is a Euclidean conformal diffeomorphism of $\source$ with a continuous extension to $\bar \source$. By Lemma~\ref{lma:boundary-fixing-conformal} and the assumption $\O=\source$, it suffices to show that $\varphi$ coincides with the identity on some open subset of $\p \source$.
We may assume without loss of generality that $\source \cap \rec=\emptyset$. Indeed, if $\source \cap \rec\neq \emptyset$, then we get from \eqref{eq:dists_agree} with $p=q\in\source\cap\rec$ that $d_{c_2}(\varphi(p),p)=d_{c_1}(p,p)=0$, which gives $\varphi(p)=p$.
Hence, $\varphi$ coincides with the identity map in $\source \cap \rec$. Since conformal maps are real analytic, it follows that $\varphi$ is the identity everywhere in $\bar \source$.

For each $q\in \rec$, define
\[
 m_j(q)=\min_{p\in \bar \source}d_{c_j}(p,q),
 \quad \text{and} \quad
 A_j(q)=\{p\in\bar \source: d_{c_j}(p,q)=m_j(q)\}.
\]
By compactness, the sets $A_j(q)$ are nonempty subsets of $\p \source$. Let $p \in A_1(q)$ and $\gamma_1$ be a $g_{c_1}$-distance-minimizing unit-speed geodesic from $q$ to $p$. Since $\gamma_1$ remains outside $\source$, where $c_1=c_2$, we have
\[
d_{c_1}(q,p)
=
\operatorname{length}_{c_1}(\gamma_1)
=
\operatorname{length}_{c_2}(\gamma_1)
\geq
d_{c_2}(p,q).
\]
On the other hand, if $\gamma_2$ is a $g_{c_2}$-distance-minimizing unit-speed geodesic from $q$ to $p$, then the first time $\gamma_2$ meets $\bar \source$ is greater than or equal to $m_1(q)$. Hence,
\[
d_{c_2}(q,p)
=
\operatorname{length}_{c_2}(\gamma_2)
\geq
m_1(q)
=
d_{c_1}(p,q).
\]
We arrive at
$
d_{c_1}(q,p)
=
d_{c_2}(p,q)
$
and by \eqref{eq:TTT_revisited}, we obtain
\[
d_{c_2}(p,q)=m_2(q),
\quad \text{and} \quad
p \in A_2(q).
\]
Thus, by interchanging the roles of $g_1$ and $g_2$, we find
\begin{equation}\label{eq:closest}
m_1(q)=m_2(q), \quad \text{and} \quad A_1(q)=A_2(q).
\end{equation}

On the other hand, \eqref{eq:dists_agree} and the bijectivity of
$\varphi$ show that
\begin{equation}\label{eq:minimizers}
 \varphi(A_1(q))=A_2(q).
\end{equation}

Since $\p\source$ is a compact smooth hypersurface in the
complete Riemannian manifold $(\R^n,g_{c_1})$, its
$g_{c_1}$-cut locus is closed and has zero Riemannian volume by
\cite[Proposition~4.8(1) and Corollary~4.13]{MantegazzaMennucci2003}.
Since $\rec$ is nonempty and open and
$\rec\cap\p\source=\emptyset$, there exists a nonempty open
set $\rec_0\subset\rec$ disjoint from this cut locus.
By \cite[Proposition~4.8(4),(7)]{MantegazzaMennucci2003},
each $q\in\rec_0$ has a unique nearest point
$\pi(q)\in\p\source$ with respect to $g_{c_1}$.
Moreover, the normal exponential map of $\p\source$
provides smooth local coordinates near each $q\in\rec_0$,
in which $\pi$ is the projection onto the boundary variable.
Thus $\pi\colon\rec_0\to\p\source$ is a smooth submersion.
Consequently, its image $\Gamma=\pi(\rec_0)$ is a relatively
open nonempty subset of $\p\source$.

For $q\in \rec_0$, equations \eqref{eq:closest}--\eqref{eq:minimizers}
now give
\[
 A_1(q)=A_2(q)=\{\pi(q)\},\qquad
 \varphi(\pi(q))=\pi(q).
\]
Thus $\varphi$ fixes every point of $\Gamma$, which completes the proof of Corollary~\ref{cor:sources_and_unknowns_agree_receivers anywhere}.

\subsection{Proof of Corollary~\ref{cor:enclosure_sources_enclosure_measurement}}
Since $\Omega$ is bounded, it follows from the choices $\rec=\source=B_e(\O,\epsilon)\setminus \bar \O$ and $T>M\diam_e(\rec)$ that $T>M\diam_e(\rec \cup \source)$. Thus, Theorem~\ref{thm:reduction} implies the identity \eqref{eq:TTT_agree}. Since $(\bar \O,g_{c_i})$ is a simple Riemannian manifold for both $i \in\{1,2\}$, our aim is to prove that
\begin{equation}
    \label{eq:boundary_dist_func_agree}
d_{c_1}(w,z)=d_{c_2}(w,z),
\quad
\text{ for all } w,z \in \p \O,
\end{equation}
where the distances must be interpreted as intrinsic distances in $(\bar \O,g_{c_i})$.
Then boundary rigidity within a conformal class of simple metrics, \cite[Theorem C]{croke1991rigidity}, implies $c_1=c_2$ in $\bar \O$.
However, we recall that \textit{a priori} the distances measured in \eqref{eq:TTT_agree} are ambient distances in $\R^n$. We therefore use assumption \eqref{eq:ext_assumptions_for_ci}, which ensures that the ambient and intrinsic $d_{c_i}$-distances agree on $\bar \O \times \bar \O$.

Let $y_1\in \source$. By \eqref{eq:TTT_agree}, there is $y_2 \in \source$ such that
\[
d_{c_1}(y_1,\dummy)|_{\rec}
=
d_{c_2}(y_2,\dummy)|_{\rec}.
\]
Since $\source=\rec$, the preceding equation yields $y_1=y_2$. Hence, we have proved that
\[
d_{c_1}(y,z)=d_{c_2}(y,z), \quad
\text{for all } y,z \in \source.
\]
Since $\p \O \subset \p \source$, equation \eqref{eq:boundary_dist_func_agree} follows from continuity of the distance function. The proof of Corollary~\ref{cor:enclosure_sources_enclosure_measurement} is complete.

\appendix

\section{When do the intrinsic and ambient distances agree?}
\label{app:exterior_metrics}
We give a sufficient condition under which assumption \eqref{eq:ext_assumptions_for_ci} of Corollary~\ref{cor:enclosure_sources_enclosure_measurement} holds. Suppose there is a continuous map $\pi\colon\R^n\to\bar\O$ that restricts to the identity on $\bar \O$ and does not increase $d_{c_i}$-lengths:
\begin{equation}
    \label{eq:curve_shortening}
\operatorname{length}_{c_i}(\pi\circ\gamma)
\leq \operatorname{length}_{c_i}(\gamma),
\qquad i\in\{1,2\},
\end{equation}
for every rectifiable curve $\gamma$.

To derive \eqref{eq:ext_assumptions_for_ci} from this assumption, take $x,y\in\bar\O$ and let $\gamma$ be an ambient
$g_{c_i}$-distance-minimizing geodesic joining them.
Since $\pi$ fixes $\bar\O$, the curve $\pi\circ\gamma$ joins
$x$ to $y$ inside $\bar\O$ and has length at most $d_{c_i}(x,y)$.
Thus the intrinsic distance in $(\bar\O,g_{c_i})$ is at most
$d_{c_i}(x,y)$, and the reverse inequality follows immediately
from the definitions. Consequently, the ambient and intrinsic
distances agree on $\bar\O\times\bar\O$. Next, we list some conditions on $c_1$ and $c_2$ that imply the existence of such a projection $\pi$.

\begin{itemize}
\item Suppose that the set $\O$ is Euclidean convex and that
\begin{equation}
    \label{eq:ext_smallness_c_i}
c_1=c_2=1 \quad\text{on }\p\O,
\qquad \text{and} \qquad
c_1=c_2\leq 1 \quad\text{in }\R^n\setminus\O.
\end{equation}
Indeed, let $\pi\colon\R^n\to\bar\O$ be the map taking a point in $\R^n$ to its closest point in $\bar \O$ with respect to the Euclidean metric. Due to the convexity of $\O$, the projection $\pi$ is continuous and fixes $\bar \O$. For every rectifiable curve $\gamma$
contained in $\R^n\setminus\O$, its projection lies on $\p\O$,
and the Euclidean nonexpansiveness of $\pi$ gives
\[
\operatorname{length}_{c_i}(\pi\circ\gamma)
=
\operatorname{length}_e(\pi\circ\gamma)
\leq
\operatorname{length}_e(\gamma)
\leq
\operatorname{length}_{c_i}(\gamma).
\]
In the first equality and the last inequality above, we used \eqref{eq:ext_smallness_c_i}. Since $\pi$ fixes $\bar\O$, we see that $\pi$ satisfies \eqref{eq:curve_shortening}.

\item
Suppose that $c_1=c_2$ in $\R^n \setminus \Omega$ and that the common exterior metric $g=g_{c_i}$ admits global boundary normal coordinates without boundary cut points. We assume that in these coordinates,
\begin{equation}
    \label{eq:boundary_normal_coordinates}
g=d\tau^2+h_\tau,
\qquad
h_\tau(v,v)\geq h_0(v,v)
\quad\text{for all }v\in T_p\p\O,
\end{equation}
where $\tau$ is the outward normal distance to $\p \O$ and $h_\tau$ is the induced metric on the smooth level sets of $\tau$, pulled back to $\p\O$ by the boundary normal coordinates.
The map $\pi$ can then be defined as the closest-point projection on the set $\R^n\setminus \O$, where it is given by $(p,\tau)\mapsto p$, and as the identity in $\O$. Hence, the inequality \eqref{eq:curve_shortening} follows from \eqref{eq:boundary_normal_coordinates}.
We refer to
\cite{MantegazzaMennucci2003} for more details about boundary normal coordinates.

\item
In fact, the projection $\pi$ needs to exist only in a sufficiently large neighborhood of $\bar\O$.
Indeed, if $\gamma\colon[0,\ell]\to\R^n$ is an ambient
unit-speed $d_{c_i}$-minimizing geodesic between points of $\bar \O$, then \eqref{eq:uniform_sound_speed} implies that
\[
\ell\leq M\diam_e(\O),
\qquad
d_{c_i}(\gamma(s),\bar\O)
\leq \min\{s,\ell-s\}
\leq \frac{M}{2}\diam_e(\O).
\]
Thus, in the previous boundary normal-coordinate formulation, it suffices for the these coordinates to remain valid without cut points and for \eqref{eq:boundary_normal_coordinates} to hold in the set $\tau^{-1}((0,r))$, where
$r>\frac{M}{2}\diam_e(\O)$.
Alternatively, if the boundary condition in \eqref{eq:ext_smallness_c_i} holds and its exterior condition holds in $B_e(\bar\O,r)\setminus\O$ for
$r>\frac{M^2}{2}\diam_e(\O)$, then we can again choose $\pi$ to be the Euclidean closest-point projection.
\end{itemize}

\section{Proofs of auxiliary lemmas}
\label{app:postponed-proofs}

We omitted the proofs of two lemmas in the main text above, as the results are not new.
As these proofs may be instructive to readers not well versed in the field, we include them in this appendix for completeness.

\begin{proof}[Proof of Lemma~\ref{lemma:globalhyp}]
    Let
    $$
    \bar \gamma(s) = (t(s), x(s)), \qquad a < s < b,
    $$
    be a piecewise $C^1$ inextendible future-directed causal curve in $\R \times \R^n$. On each $C^1$ segment, the future-directedness condition implies $\dot{t}(s)>0$. This guarantees that $t(s)$ is a strictly increasing function of $s$. So we may reparametrize $\bar \gamma$ with respect to $t$ to get the piecewise $C^1$ curve
    $$
    \widetilde{\gamma}(t) = \bar \gamma(s(t)) = (t, \widetilde{x}(t)), \qquad \widetilde{a} < t < \widetilde{b}.
    $$
    Next, the causality assumption
    $
    -1+|\dot{\widetilde{x}}(t)|_{g_c}^2 \leq 0
    $
    implies
    $
    |\dot{\widetilde{x}}(t)|_{g_c}\leq 1
    $
    at every point where $\tilde x$ is differentiable. Consequently,
    \begin{equation}\label{eq:gammacauchy}
        d_c(\widetilde{x}(t_1), \widetilde{x}(t_2)) \leq \int_{t_1}^{t_2} |\dot{\widetilde{x}}(t)|_{g_c}\, dt \leq |t_2-t_1| \qquad \text{whenever } \widetilde{a} < t_1 \leq t_2 <\widetilde{b}.
    \end{equation}

    We claim that $\widetilde{a}=-\infty$ and $\widetilde{b}=+\infty$. Indeed, suppose on the contrary that $\widetilde{b}<+\infty$. Then \eqref{eq:gammacauchy} implies that $(\widetilde{x}(t))$ is Cauchy as $t\to \widetilde{b}$. Since $(\R^n,g_c)$ is complete by the assumption \eqref{eq:uniform_sound_speed}, there exists $x_b\in \R^n$ such that $\widetilde{x}(t)\to x_b$ as $t\to \widetilde{b}$.
    This implies that $\widetilde{\gamma}$ can be extended past $\widetilde{b}$ by the timelike curve $t \mapsto (t, x_b)$, which contradicts the assumption that $\bar \gamma$ is inextendible. So we conclude that $\widetilde{b} = +\infty$. Similar arguments show that $\widetilde{a}=-\infty$. It follows therefore that for all $t\in \R$, $\bar \gamma$ meets $\Sigma_t = \{t\}\times \R^n$ exactly once. The argument for past-directed causal curves is analogous.
\end{proof}

\begin{proof}[Proof of Lemma~\ref{lemma:projection_to_null_geos}]
Let
$
\Gamma(s)=\bigl(t(s),x(s),\tau(s),\xi(s)\bigr)
$
be an integral curve of the Hamiltonian vector field $H_p$. That is,
\begin{equation}
\label{eq:int_curve}
\dot \Gamma(s)=H_p(\Gamma(s)).
\end{equation}
By \eqref{eq:principalsymb} and \eqref{eq:Hamiltonian-vector-field}, the Hamiltonian vector field takes the form
\[
H_p
=
-2\tau\,\partial_t
+
2c(x)^2\sum_{j=1}^n\xi_j\partial_{x_j}
-
|\xi|^2\sum_{j=1}^n
\partial_{x_j}\bigl(c(x)^2\bigr)\partial_{\xi_j}.
\]
Thus, \eqref{eq:int_curve} gives the following Hamilton equations:
\begin{equation}
\label{eq:Hamilton-equations_1}
\begin{aligned}
\dot{t}
&=
\partial_\tau p
=
-2\tau,
\qquad
\dot{x}
=
\partial_\xi p
=
2c(x)^2\xi,
\\
\dot{\tau}
&=
-\partial_t p
=
0,
\qquad
\dot{\xi}
=
-\partial_xp
=
-2c(x)|\xi|^2\nabla c(x).
\end{aligned}
\end{equation}

We show that the curve $(t(s),x(s))$ satisfies the geodesic equation of $\bar g$. The time component of the geodesic equation follows immediately from
\eqref{eq:Hamilton-equations_1}, since
\[
\ddot{t}=-2\dot{\tau}=0.
\]
To obtain the spatial geodesic equation, we differentiate the identity
\[
\dot{x}=2c(x)^2\xi,
\]
which gives
\[
\ddot{x}
=
4c(x)\bigl(\nabla c(x)\cdot\dot{x}\bigr)\xi
+
2c(x)^2\dot{\xi}.
\]
Using \eqref{eq:Hamilton-equations_1}, we obtain
\[
\ddot{x}
=
\frac{2}{c(x)}
\bigl(\nabla c(x)\cdot\dot{x}\bigr)\dot{x}
-
\frac{|\dot{x}|^2}{c(x)}\nabla c(x).
\]
This can be written equivalently as the geodesic equation
\begin{equation}
\label{eq:spatial-geodesic-equation}
\ddot{x}
-
2\bigl(\nabla\log c(x)\cdot\dot{x}\bigr)\dot{x}
+
|\dot{x}|^2\nabla\log c(x)
=
0,
\end{equation}
for the conformally Euclidean metric $g_c$. Thus the base projection $\bar\gamma(s)=\bigl(t(s),x(s)\bigr)$ of $\Gamma(s)$ is a geodesic of $\bar g$.

Using \eqref{eq:Hamilton-equations_1} again, we have
\begin{equation*}
\overline{g}(\dot{\bar\gamma},\dot{\bar\gamma})
=
-\dot{t}^{\,2}
+
c(x)^{-2}|\dot{x}|^2
=
4p(\Gamma)=0,
\end{equation*}
where the last equation holds since $\Gamma$ is a null bicharacteristic. Therefore, the base projection of a null bicharacteristic of $p$ is a null geodesic of $\overline{g}$.
\end{proof}

\section{A Fourier integral operator proof of Proposition~\ref{prop:llc}} \label{app:FIO-proof}

Here we give an alternative proof of Proposition~\ref{prop:llc} using an explicit local Fourier integral operator parametrix for the Cauchy problem. We first collect the additional material on FIOs needed for that argument.

\subsection{The local FIO parametrix}

We first recall the local Fourier integral representation of the solution operators for the Cauchy problem \eqref{general-ivp}. Following \cite[Ch.~6]{Gr-Sj}, we say that a differential operator $Q$ of order $m$ on $\R_t\times\R^n_x$ with principal symbol $q(t,x,\tau,\xi)$ is \emph{strictly hyperbolic} with respect to $t$ if for every $(t,x,\xi)$ with $\xi\neq 0$, the equation $q(t,x,\tau,\xi)=0$ has exactly $m$ distinct real roots in $\tau$. It is easy to see from \eqref{eq:principalsymb} that $p(t,x,\tau,\xi)=0$ if and only if
  \begin{equation}\label{eq:characteristicroots}
     \tau = \lambda_{\pm}(t,x,\xi) := \mp c(x)|\xi|.
  \end{equation}
The sign convention is chosen so that $\{\tau = \lambda_{+}(t,x,\xi)\}$ corresponds to the forward branch in time, since the Hamiltonian flow of $p(t,x,\tau,\xi)$ on this branch satisfies $\dot{t} = \partial_\tau p = -2\tau >0$. Since $c(x)>0$, these roots are real and distinct for $\xi \neq 0$. Thus $P_{c,\alpha}$ is strictly hyperbolic with respect to $t$. The local parametrix construction in \cite[Ch.~6, in particular, Theorem~6.2 and its proof]{Gr-Sj} yields the following consequence.

\begin{Lemma}\label{lemma:parametrix}
Let $U_0 \subset \R^n$ be a relatively compact open subset. Then there exists an open set $U \subset \R\times \R^n$ with $U \cap (\{0\}\times \R^n) = \{0\}\times U_0$, for which we can construct operators
$$E_k: \mathcal{E}'(U_0) \to \mathcal{D}'(U), \qquad k=0,1, $$
such that for all $f,g\in \mathcal{E}'(U_0)$, the solution $u$ of \eqref{general-ivp} can be expressed in $U$ as
\begin{equation*}
u = E_0f + E_1g + r,
\end{equation*}
where $r \in C^\infty(U)$. In fact, $E_0,E_1$ are sums of Fourier integral operators:
\begin{equation}\label{eq:FIOrepresentation}
E_k h(t,x) = \sum_{\pm} E_{k,\pm}h(t,x) = \frac{1}{(2\pi)^n} \sum_{\pm}\int_{\R^n} e^{i\varphi_{\pm}(t,x,\eta)}a_{k,\pm}(t,x,\eta)\widehat{h}(\eta) \, d\eta, \qquad k =0,1, \end{equation}
where $\varphi_{\pm}$ are real-valued phase functions (in the sense of \cite[Def.~1.10]{Gr-Sj}), and $a_{k,\pm}(t,x,\eta)$ are symbols of order $-k$ and type $(1,0)$ (in the sense of \cite[Def.~1.1]{Gr-Sj}).
\end{Lemma}

The phase functions $\varphi_{\pm}$ correspond to the two roots $\tau = \lambda_{\pm}(t,x,\xi)$ of $p(t,x,\tau,\xi)=0$ in the sense of satisfying \cite[Eq.~$(6.5)_\nu$ and (6.6)]{Gr-Sj}. In the present case, these equations become
\begin{equation}\label{eq:eikonal}
    \partial_t\varphi_{\pm} = \mp c(x)|\nabla_x \varphi_{\pm}|, \qquad \varphi_{\pm}(0,x,\eta) = x\cdot \eta,
\end{equation}
which can be uniquely solved in a suitably chosen domain $U$ using standard Hamilton--Jacobi theory (see, for instance, \cite[Ch.~5]{Gr-Sj}).

The amplitudes $a_{k,\pm}(t,x,\eta)$ are classical symbols of order $-k$, i.e., they admit asymptotic expansions
\begin{equation} \label{eq:asymptoticexpansion}
    a_{k,\pm} \sim \sum_{j=0}^\infty a_{k,\pm}^{-k-j}(t,x,\eta),
\end{equation}
where $a_{k,\pm}^{-k-j}$ are positively homogeneous of degree $-k-j$ in $\eta$. As usual, these terms are required to be homogeneous only away from $\eta=0$; multiplying the terms by a cut-off near $\eta=0$ may be needed to ensure global smoothness, but this changes the corresponding operators only by smoothing terms.

The leading-order terms $a^{-k}_{k,\pm}$, $k=0,1$, are found by imposing the Cauchy conditions on $E_0,E_1$ at the principal-symbol level and then solving the corresponding transport equations \cite[Eq.~(6.7)]{Gr-Sj}. For $E_0$, the Cauchy conditions
$$E_0|_{t=0} = \id \quad \text{and} \quad \partial_tE_0|_{t=0} = 0 \quad \text{modulo smoothing terms}$$
give, by equating the principal symbols of both sides,
$$a^0_{0,+}(0,x,\eta)+a^0_{0,-}(0,x,\eta) = 1,$$
and
$$-ic(x)|\eta|a^0_{0,+}(0,x,\eta)+ ic(x)|\eta|a^0_{0,-}(0,x,\eta) = 0.$$
Here we have used the fact that by \eqref{eq:eikonal}, $\partial_t\varphi_+(0,x,\eta) = -c(x)|\nabla_x(x\cdot\eta)| = -c(x)|\eta|$. Thus,
\begin{equation}\label{eq:principalamplitudesE0}
    a^0_{0,+}(0,x,\eta) = a^0_{0,-}(0,x,\eta) = \frac{1}{2}.
\end{equation}
Similarly, for $E_1$, the Cauchy conditions
$$E_1|_{t=0} = 0 \quad \text{and}\quad \partial_tE_1|_{t=0} = \id \quad \text{modulo smoothing terms}$$
give
$$a^{-1}_{1,+}(0,x,\eta)+a^{-1}_{1,-}(0,x,\eta) = 0,$$
and
$$-ic(x)|\eta|a^{-1}_{1,+}(0,x,\eta)+ ic(x)|\eta|a^{-1}_{1,-}(0,x,\eta) = 1.$$
Therefore,
\begin{equation}\label{eq:principalamplitudesE1}
    a^{-1}_{1,\pm}(0,x,\eta) = \pm \frac{i}{2c(x)|\eta|} \qquad \text{for }\eta \text{ away from }0.
\end{equation}
In particular, \eqref{eq:principalamplitudesE0} will be used in the proof of Proposition~\ref{prop:llc}. The remaining lower-order terms in \eqref{eq:asymptoticexpansion} can be found by successively solving similar transport equations; see \cite[Ch.~6]{Gr-Sj} for details.

\subsection{Canonical relations and wave front set mapping}\label{subsec:FIOprelim}

We next recall the notion of canonical relations of FIOs and their relation to the wave front set mapping properties of the corresponding FIOs. We largely follow the notation and terminology of \cite{HormFIO}.

Let $X \subset \R^{n_X}$ and $Y \subset \R^{n_Y}$ be open subsets. Consider a Fourier integral operator $A:\mathcal{E}'(Y)\to \mathcal{D}'(X)$ with Schwartz kernel of the form (cf. \cite[Theorem~1.4.1]{HormFIO})
$$K_A(x,y) = \frac{1}{(2\pi)^N}\int_{\R^N} e^{i\Phi(x,y,\theta)}a(x,y,\theta)\, d\theta,$$
where $\Phi$ is an operator phase function on $X\times Y\times\R^N$ in the sense of \cite[Def.~1.4.4]{HormFIO}, and $a$ is a symbol on the same space of type $(1,0)$. If $\Phi$ is nondegenerate (i.e., the differentials $d\left(\partial\Phi/\partial \theta_j \right)$ are linearly independent whenever $\partial_\theta \Phi = 0$), it is associated with a Lagrangian submanifold of $T^*(X\times Y) \setminus 0$ given by
\begin{equation}\label{eq:associatedlagrangian}
    \Lambda_\Phi = \left\{(x,y,\partial_x\Phi, \partial_y \Phi) \, : \, \partial_\theta \Phi = 0 \right\}.
\end{equation}
This is precisely the image of the map described in \cite[Eq.~(3.1.1)--(3.1.2)]{HormFIO}. The corresponding homogeneous canonical relation (see \cite[Def.~4.1.2]{HormFIO}) is
$$C_\Phi = \left\{ \left((x, \partial_x \Phi), (y, -\partial_y \Phi)\right) \, : \, \partial_\theta \Phi = 0 \right\} \subset (T^*X \setminus 0)\times (T^*Y \setminus 0).$$
We will need the wave front set mapping properties for such operators. Assume that $a(x,y,\theta)$ vanishes near the zero section; this can be arranged without changing $A$ modulo smoothing terms. Then by \cite[Prop.~2.5.7]{HormFIO}, we have
\begin{equation}\label{eq:WFKA}
    WF(K_A) \subset \Lambda_\Phi.
\end{equation}
By the definition of operator phase functions, $\Phi(x,y,\theta)$ has no critical points in $(y,\theta)$ (resp., $(x,\theta)$) for each $x$ (resp., $y$). Consequently, $\partial_x\Phi$ and $\partial_y\Phi$ are both nonvanishing on the critical set $\{\partial_\theta \Phi = 0\}$. Together with \eqref{eq:WFKA}, this implies that the sets
\begin{align*}
    WF_X(K_A) &:= \{(x,\xi)\in T^*X \, : \, (x,y,\xi,0) \in WF(K_A) \text{ for some }y\in Y\}, \quad \text{and} \\
    WF_Y(K_A) &:= \{(y,\eta) \in T^*Y \, : \, (x,y,0,\eta) \in WF(K_A) \text{ for some }x \in X\}
\end{align*}
are both empty. So it follows from \cite[Theorem~2.5.14]{HormFIO} that for any $u \in \mathcal{E}'(Y)$,
\begin{equation}\label{eq:WFmappingforFIO}
   WF(Au) \subset C_\Phi \circ WF(u),
\end{equation}
where the right-hand side is interpreted as the action of a relation on a set, i.e.,
$$ C_\Phi\circ \Gamma = \left\{(x,\xi)\in T^*X\setminus 0: \text{ there exists } (y,\eta)\in \Gamma \text{ with } \left((x,\xi);(y,\eta)\right)\in C_\Phi \right\}. $$

Now consider the FIO parametrices $E_{0,\pm}$ defined in \eqref{eq:FIOrepresentation}. Their Schwartz kernels $K_{0,\pm}\in \mathcal{D}'(U\times U_0)$ are given by
$$K_{0,\pm}(t,x,y) = \frac{1}{(2\pi)^n}\int_{\R^n}e^{i\Phi_{\pm}(t,x,y,\eta)}a_{0,\pm}(t,x,\eta) \, d\eta, \qquad \text{where }\Phi_{\pm} := \varphi_\pm(t,x,\eta)-y\cdot \eta.$$
The differentials $d(\partial \Phi_{\pm}/ \partial \eta_j)$ are linearly independent since the matrix $\left(\partial^2\Phi_{\pm}/\partial y_j \partial \eta_k\right) = -\id$, so $\Phi_\pm$ are indeed nondegenerate phase functions. We now compute their canonical relations. Imposing the stationary condition on $\Phi_{\pm}$ in $\eta$, we get
$$\partial_\eta \Phi_{\pm} = 0 \qquad \text{if and only if} \qquad y = \partial_\eta\varphi_{\pm}(t,x,\eta).$$
Therefore, the corresponding canonical relations are
\begin{equation}\label{eq:canonicalrelations}
    C_{\pm} = \left\{ \left((t,x,\partial_t\varphi_{\pm}, \partial_x\varphi_{\pm});(y,\eta)\right) \, : \, y = \partial_\eta\varphi_{\pm}\right\},
\end{equation}
 and for any $u \in \mathcal{E}'(U_0)$,
\begin{equation}\label{WFmappingforE}
    WF(E_{0,\pm}u) \subset C_{\pm}\circ WF(u).
\end{equation}
Note in particular that at $t=0$,
\begin{equation}\label{eq:canonicalrelationat0}
    \left((0,y,\mp c(y)|\eta|,\eta);(y,\eta)\right) \in C_{\pm}\qquad \text{for all }(y,\eta) \in T^*U_0 \setminus 0.
\end{equation}

\subsection{Proof of the wave front set characterization}

\begin{proof}[Alternative proof of Proposition~\ref{prop:llc}]
    Since $A \neq 0$, the wave front set of $u^{y,A}=Au^{y,1}$ coincides with the wave front set of $u^{y,1}$. So we assume without loss of generality that $A = 1$ and write $u^y = u^{y,1}$. We first prove the inclusion
    \begin{equation}\label{eq:global-wf-cone}
    WF(u^y) \subset LLC_{g_c}(0,y).
    \end{equation}
    By Lemma~\ref{lemma:parametrix}, there exists a relatively compact neighborhood $U$ of $(0,y)$ in which
    \begin{equation}\label{eq:parametrixinU}
    u^y = E_{0,+}\delta_y +E_{0,-}\delta_y\qquad \text{modulo }C^\infty(U).
    \end{equation}
    Choose $\eps>0$ small enough that
    $$\{(t,x) \, : \, |t| < \eps, \, d_c(x,y) \leq |t|\} \subset U.$$
    Then by Lemma~\ref{sharp},
    $$\supp(u^y) \cap \{|t| < \eps\} \subset U.$$
    Hence every element of $WF(u^y)$ over the strip $\{|t|<\eps\}$ lies in $T^*U$. So \eqref{eq:parametrixinU} together with the wave front set mapping property \eqref{WFmappingforE} gives
    \begin{equation}\label{eq:WFupperboundnearU}
        WF(u^y)\cap \{|t|<\eps\} \subset (C_+\circ WF(\delta_y))\cup (C_{-}\circ WF(\delta_y)),
    \end{equation}
    where $C_\pm$ are the canonical relations defined in \eqref{eq:canonicalrelations}. Since
    $$WF(\delta_y) = \{(y,\eta) \in T^*_y\R^n \, : \, \eta \neq 0 \},$$
    we have
    $$C_{\pm}\circ WF(\delta_y) = \left\{\left(t,x,\partial_t\varphi_\pm(t,x,\eta),\partial_x\varphi_\pm(t,x,\eta)\right) \, : \, \partial_\eta\varphi_\pm(t,x,\eta) = y, \, \eta \neq 0 \right\}.$$
    At $t=0$, the eikonal equations and their initial conditions in \eqref{eq:eikonal} imply
    $$\partial_t\varphi_\pm(0,x,\eta) = \lambda_\pm(x,\eta) = \mp c(x)|\eta|,\qquad \varphi_\pm(0,x,\eta) = x\cdot \eta, \qquad \partial_\eta\varphi_{\pm}(0,x,\eta) = x,$$
    which show that
    $$\Lambda^0_\pm :=\{(0,y, \lambda_{\pm}(y,\eta), \eta) \, : \, \eta \neq 0\} \subset C_\pm \circ WF(\delta_y).$$
    We claim that near $t=0$, $C_\pm \circ WF(\delta_y)$ are precisely the flow-outs of $\Lambda^0_\pm$ under the Hamiltonian vector field of $p= -\tau^2+c(x)^2|\xi|^2$. Indeed, let
		$$q_\pm(t,x,\tau,\xi)=\tau-\lambda_\pm(x,\xi) = \tau\pm c(x)|\xi|.$$
	Then $p=-q_+q_-$, and the eikonal equations \eqref{eq:eikonal} become
	$$q_\pm(t,x,\partial_t\varphi_\pm,\partial_x\varphi_\pm)=0.$$
	  Now fix $\eta\neq 0$. The equation $\partial_\eta\varphi_\pm(t,x,\eta)=y$ can be solved locally for $x=x_\pm(t)$ in terms of $t$, $y$, and $\eta$ near $t=0$, since
	$$\partial_x\partial_\eta\varphi_\pm(0,y,\eta)=I.$$
	Set
	$$\xi_\pm(t)=\partial_x\varphi_\pm(t,x_\pm(t),\eta), \qquad
		\tau_\pm(t)=\partial_t\varphi_\pm(t,x_\pm(t),\eta).$$
	Differentiating the identity
	$$\partial_\eta\varphi_\pm(t,x_\pm(t),\eta)=y$$
	with respect to $t$ and using the eikonal equations gives
	$$ \dot x_\pm(t) = -\partial_\xi\lambda_\pm(x_\pm(t),\xi_\pm(t)), \qquad \text{and} \qquad \dot\xi_\pm(t) = \partial_x\lambda_\pm(x_\pm(t),\xi_\pm(t)).$$
	Thus
	$$(t,x_\pm(t),\tau_\pm(t),\xi_\pm(t))$$
	is an integral curve of
	$$ H_{q_\pm} = \partial_t - \partial_\xi\lambda_\pm\cdot\partial_x + \partial_x\lambda_\pm\cdot\partial_\xi$$
	with initial point $(0,y,\lambda_\pm(y,\eta),\eta)$. Since $H_p$ is a nonzero scalar multiple of $H_{q_\pm}$ on $\Sigma_\pm=\{q_\pm=0\}$, these curves are also the null bicharacteristics of $p$, up to reparametrization. In other words, we have shown that, after possibly shrinking $U$,
    $$(C_+\circ WF(\delta_y))\cup (C_-\circ WF(\delta_y))$$
    is the union of null bicharacteristics of $p$ in $T^*U$ passing over $(0,y)$. Combined with \eqref{eq:WFupperboundnearU}, we get
    \begin{equation}\label{eq:local-wf-cone}
        WF(u^y)\cap \{|t|<\eps\} \subset LLC_{g_c}(0,y).
    \end{equation}

    We now globalize this local inclusion. Let $\rho \in WF(u^y)$, and let $\Gamma$ be the maximally extended bicharacteristic of $p$ through $\rho$. By Lemma~\ref{propa}, the image of $\Gamma$ is contained in $WF(u^y)$. Its projection to $\R\times\R^n$ is a maximally extended null geodesic of $\bar g$ by Lemma~\ref{lemma:projection_to_null_geos}. So Lemma~\ref{lemma:globalhyp} implies that $\Gamma$ meets the Cauchy hypersurface $\{t=0\}$, and hence the strip $\{|t|< \eps\}$. Choose $\rho_0 \in \Gamma$ over this strip. By \eqref{eq:local-wf-cone}, $\rho_0 \in LLC_{g_c}(0,y)$. Since $LLC_{g_c}(0,y)$ is invariant under the Hamiltonian flow of $p$, the entire image of the bicharacteristic $\Gamma$ is contained in $LLC_{g_c}(0,y)$. In particular, $\rho \in LLC_{g_c}(0,y)$, which proves \eqref{eq:global-wf-cone}.

    It remains to prove the reverse inclusion. Let $\rho_0 = (0,y,\tau_0,\eta_0)$ be an arbitrary element of $LLC_{g_c}(0,y)\cap T^*_{(0,y)}(\R\times\R^n)$. Suppose it belongs to the forward light cone $\Sigma_{+} = \{(t,x,\tau,\xi)\, : \, \tau = -c(x)|\xi|\}$, so that $\tau_0 = -c(y)|\eta_0|$ (the argument for the backward case is similar). Applying $E_{0,+}$ to $\delta_y$, we get
    $$E_{0,+}\delta_y(t,x) = \frac{1}{(2\pi)^n}\int_{\R^n} e^{i(\varphi_+(t,x,\eta)-y\cdot \eta)}a_{0,+}(t,x,\eta)\, d\eta.$$
    The associated Lagrangian manifold (see \eqref{eq:associatedlagrangian} in Subsection~\ref{subsec:FIOprelim}) is given by
    $$\Lambda_{y,+} := \left\{(t, x, \partial_t\varphi_+(t,x,\eta), \partial_x \varphi_+(t,x,\eta))\, : \, \partial_\eta\varphi_+(t,x,\eta) = y \right\}.$$
    At time $t=0$, the stationary condition $\partial_\eta\varphi_+(t,x,\eta) = y$ reduces to $x=y$. Moreover, the eikonal equation and the initial condition for $\varphi_+$ give
    $$\partial_t \varphi_+(0,y,\eta_0) = -c(y)|\eta_0| = \tau_0, \qquad \partial_x \varphi_+(0,y,\eta_0) = \eta_0.$$
    Hence
    $$\rho_0 = (0,y,\tau_0, \eta_0) \in \Lambda_{y,+}.$$

    By the principal symbol characterization of Lagrangian distributions \cite[Sec.~3.2 and Theorem~3.2.6]{HormFIO}, the principal symbol of $E_{0,+}\delta_y$ at $\rho_0 \in \Lambda_{y,+}$ is represented, up to a nonzero normalization factor, by the principal amplitude $a_{0,+}^0(0,y,\eta_0)$. If $\rho_0 \notin WF(E_{0,+}\delta_y)$, then the principal symbol of $E_{0,+}\delta_y$ drops to a lower order near $\rho_0$ by \cite[Theorem~3.2.6]{HormFIO}.
    However, we have already seen in \eqref{eq:principalamplitudesE0} that
    $$a_{0,+}^0(0,y,\eta) = \frac{1}{2} \qquad \text{for all } y \in \R^n, \, \eta \neq 0.$$
    So we conclude that $\rho_0 \in WF(E_{0,+}\delta_y)$. Moreover,
    $$WF(E_{0,-}\delta_y) \subset C_{-}\circ WF(\delta_y) \subset \{(t,x,\tau,\xi)\, : \, \tau = +c(x)|\xi|, \, \xi \neq 0 \},$$
    so $E_{0,-}\delta_y$ is microlocally smooth near $\rho_0$. Since
    $$u^y = E_{0,+}\delta_y + E_{0,-}\delta_y \qquad \text{mod $C^\infty$}$$
    near $(0,y)$, we conclude that $\rho_0 \in WF(u^y)$. Finally, let $\rho \in LLC_{g_c}(0,y)$ be arbitrary, and let $\Gamma$ be the maximally extended null bicharacteristic of $p$ through $\rho$. By assumption, $\Gamma$ passes over $(0,y)$. Let $\rho_0$ be the covector in $\Gamma$ over $(0,y)$. Then by the preceding argument, $\rho_0 \in WF(u^y)$. Now Lemma~\ref{propa} implies that $\rho \in WF(u^y)$ as well. Thus we have proved the reverse inclusion
    \begin{equation} \label{eq:reverse-inc}
        LLC_{g_c}(0,y) \subset WF(u^y).
    \end{equation}
    Combining \eqref{eq:global-wf-cone} and \eqref{eq:reverse-inc}, we get
    \begin{equation}WF(u^y) = LLC_{g_c}(0,y),\end{equation}
    as claimed.
\end{proof}

\bibliography{references2}
\bibliographystyle{plain}

\end{document}